\documentclass[12pt,centertags,oneside]{article}
\usepackage{amsmath,amstext,amsthm,amscd,typearea}
\usepackage{amssymb}
\usepackage{a4wide}
\usepackage[mathscr]{eucal}
\usepackage{mathrsfs}
\usepackage{typearea}
\usepackage{charter}
\usepackage{pdfsync}
\usepackage{url}
\usepackage[T1]{fontenc}

\usepackage{hyperref}

\usepackage{xcolor}

\usepackage[a4paper,width=16.2cm,top=3cm,bottom=3cm]{geometry}

\numberwithin{equation}{section}

\newtheorem{theorem}{Theorem}[section]
\newtheorem{definition}[theorem]{Definition}
\newtheorem{proposition}[theorem]{Proposition}
\newtheorem{corollary}[theorem]{Corollary}
\newtheorem{lemma}[theorem]{Lemma}

\newtheorem{example}[theorem]{Example}

\newcommand{\supp}{{\rm Supp}}

\newcommand{\dist}{\mathop{\mathrm{dist}}\nolimits}

\newcommand{\ddc}{dd^c}
\newcommand{\dc}{d^c}

\newcommand{\PSH}{{\rm PSH}}

\newcommand{\capK}{\text{cap}}

\newcommand{\B}{\mathbb{B}}
\newcommand{\C}{\mathbb{C}}

\newcommand{\N}{\mathbb{N}}

\newcommand{\R}{\mathbb{R}}

\title{\bf  Uniform diameter estimates for K\"ahler metrics}
\providecommand{\keywords}[1]{\textbf{\textit{Keywords:}} #1}
\providecommand{\subject}[1]{\textbf{\textit{Mathematics Subject Classification 2010:}} #1}

\author{Duc-Viet Vu}

\newcommand{\Addresses}{{
		\bigskip
		\footnotesize
		\textsc{Duc-Viet Vu, University of Cologne, Division of Mathematics, Department of Mathematics and Computer Science, Weyertal 86-90, 50931, K\"oln.}
		\noindent
		\par\nopagebreak
		\noindent
		\textit{E-mail address}: \texttt{dvu@uni-koeln.de}
}}

\date{\today}

\begin{document}

\maketitle

\begin{abstract} We prove a uniform diameter estimate and a uniform local non-collapsing of volumes for a large family of K\"ahler metrics generalizing those obtained recently by Guo-Phong-Song-Sturm.  We treat also similar questions in the singular setting. 
\end{abstract}
\noindent
, {Diameter}, {Closed positive current}, {Dinh-Sibony-Sobolev space}, {Gromov-Hausdorff topology}, {K\"ahler-Ricci flow}.
\\

\noindent
\subject{32U15}, {32Q15}, {32S99}, {53C23}.



\section{Introduction}

Let $(X,\omega_X)$ be a compact K\"ahler manifold of dimension $n$. Studying the convergence of the metric spaces $(X,\omega)$ as $\omega$ varies in a certain family of K\"ahler forms has been a central topic in recent developments in K\"ahler geometry. Two important instances of such questions are the degeneration of Ricci-flat (or K\"ahler-Einstein) metrics  and the convergence of K\"ahler-Ricci flows; see, e.g., \cite{Donaldson-Sun,DonaldsonSun2,Liu-Szekelyhidi,Liu-Szekelyhidi2,Tian-survey-KE,Tosatti-survey,Tosatti-KEflow,Song-Tian-canonicalmeasure}.  In order to apply Gromov's precompactness theorem for Gromov-Hausdorff convergence of metric spaces, it is crucial to have, among other things, a uniform bound on the diameters of metric spaces in consideration. This paper concerns such a bound for a quite large family of K\"ahler metrics of Monge-Amp\`ere type.   

The problem has a long history dated back to classical results in Riemannian geometry. There were several ad hoc uniform diameter bounds available in the literature in dealing with degenerations of K\"ahler metrics such as Ricci-flat ones or K\"ahler-Ricci flows. A very general treatment was found recently by Guo-Phong-Song-Sturm  \cite{Guo-Phong-Song-Sturm} developing further the method in \cite{Guo-Phong-Sturm-green,Guo-Phong-Tong}. Let us now recall their result. Let $\mathcal{K}(X)$ be the space of smooth K\"ahler metrics on $X$.

Let $A,K$ be positive constants and let $p_0>n$ be a constant. For $\omega \in \mathcal{K}(X)$, we denote
$$V_\omega:= \int_X \omega^n, \quad \mathcal{N}_{X, \omega_X,p_0}(\omega):= \frac{1}{V_\omega}\int_X \bigg|\log \big(V_\omega^{-1} \frac{\omega^n}{\omega_X^n} \big) \bigg|^{p_0} \omega^n.$$
We define
$$\mathcal{V}(X,\omega_X,n,A,p_0,K):= \{ \omega \in \mathcal{K}(X): \{\omega\} \cdot \{\omega_X\}^{n-1} \le A, \, \mathcal{N}_{X, \omega_X, p_0}(\omega) \le K \},$$
where $\{\omega\}, \{\omega_X\}$ denote the cohomology class of $\omega$, $\omega_X$ respectively. One observes that the condition 
$$\{\omega\} \cdot \{\omega_X\}^{n-1} \le A$$
 is equivalent to saying that $\{\omega\}$ belongs a fixed compact subset in  $H^{1,1}(X,\R)$.  And if $V_\omega^{-1}\omega^n =f \omega_X^n$ for $f \in L^1 \log L^{p_0}(X, \omega_X^n)$, then 
$$\mathcal{N}_{X, \omega_X, p_0}(\omega)=\|f\|_{L^1 \log L^{p_0}}:= \int_X f |\log f|^{p_0} \omega_X^n.$$
Let $\gamma$ be a non-negative continuous function on $X$ satisfying that the Hausforff dimension of $\{\gamma=0\}$ is strictly less than $2n-1$. Let 
$$\mathcal{W}(X, \omega_X,n,A,p_0,K,\gamma):= \{\omega \in \mathcal{V}(X, \omega_X,n,A,p_0,K): V_\omega^{-1} \omega^n \ge \gamma \omega_X^n\}.$$
Roughly speaking, $\mathcal{W}(X, \omega_X,n,A,p_0,K,\gamma)$ is a subset of $\mathcal{V}(X, \omega_X,n,A,p_0,K)$ containing K\"ahler metrics whose Monge-Amp\`ere is bounded from below uniformly away from $0$ outside a  (fixed) subset in $X$ of Hausdorff dimension $<2n-1$. One should notice that the family $\mathcal{V}(X, \omega_X,n,A,p_0,K)$ is much larger than $\mathcal{W}(X, \omega_X,n,A,p_0,K,\gamma)$; see Example \ref{ex-no-nonvanishing} at the end of the paper.

It was proved in \cite[Theorem 1.1]{Guo-Phong-Song-Sturm} that the diameter of $(X,\omega)$ for 
$$\omega \in \mathcal{W}(X, \omega_X,n,A,p_0,K,\gamma) \quad (p_0>n)$$
 is uniformly bounded, and a uniform local non-collapsing estimate for volume of $\omega$ holds as well. This extends many previous results, especially, some of those in \cite{Jian-Song}.   In particular, one obtains a version of Gromov's precompactness theorem for the family  $\mathcal{W}(X, \omega_X,n,A,p_0,K,\gamma)$ (see \cite[Theorem 1.2]{Guo-Phong-Song-Sturm}). These results were extended recently in \cite{Guo-Phong-Song-Sturm2} to the case where $X$ is singular but only for a subfamily of $\mathcal{W}(X, \omega_X,n,A,p_0,K,\gamma)$ consisting of metrics whose volumes are uniformly non-collapsing (see \cite[Theorem 3.1]{Guo-Phong-Song-Sturm2} for a precise statement).  One of the notable features of these results is that they do not require any assumption on a lower bound for Ricci curvature which is usually assumed in the classical convergence theory of Riemannian manifolds. 

Under assumptions on lower bounds for Ricci curvature, diameter estimates and local non-collapsing of volumes were also established in \cite{Fu-Guo-Song-geometricestimates,Guo-Song-localnoncollapsing}  if a lower bounded for Ricci curvature is assumed; we refer  also to \cite{GGZ-logcontiu} for some generalizations, and to  \cite{Tosatti-collapsing,Rong-Zhang} for the case of Ricci-flat metrics. 
However the main issue is that the condition on a lower bound for Ricci curvature is difficult to verify in applications (such a bound does not even hold in some cases; see \cite[Section 5]{GGZ-logcontiu}).

We underline that if we restrict ourselves to a subset of $\mathcal{V}(X,\omega_X,n,A,p_0,K)$ containing $\omega$ in a given compact subset of the K\"ahler cone in $X$, then we can say much more about the distance function of $\omega$: it is $\log$ continuous with respect to the distance induced by $\omega_X$; see \cite{Guo-Phong-Tong-Wang,GGZ-logcontiu,YangLi,Vu-log-diameter}. This $\log$ continuity property also holds for the distance function of Ricci-flat metrics in semi-positive big cohomology class; see \cite{Do-Vu-log-continuity}.

In order to treat at the same time the singular setting, we  introduce now a larger class of metrics.   Let $\tilde{\mathcal{K}}(X)$ be the set of closed positive $(1,1)$-currents $\omega$ of bounded potentials on $X$ such that $\omega$ is a smooth K\"ahler form on an open Zariski dense subset in $X$ and belong to  a big and semi-positive cohomology class in $X$. Recall that a cohomology class is said to be semi-positive if it contains a smooth semi-positive closed form.  For $\omega \in \tilde{\mathcal{K}}(X)$, the quantity $\mathcal{N}_{X,\omega_X,p_0}(\omega)$ still makes sense. We put
$$\tilde{\mathcal{V}}(X,\omega_X,n,A,p_0,K):= \{ \omega \in \tilde{\mathcal{K}}(X): \{\omega\} \cdot \{\omega_X\}^{n-1} \le A, \, \mathcal{N}_{X, \omega_X, p_0}(\omega) \le K \}.$$
Observe that $\tilde{\mathcal{V}}(X,\omega_X,n,A,p_0,K)$ contains $\mathcal{V}(X,\omega_X,n,A,p_0,K)$. Moreover if $\theta$ is a smooth closed $(1,1)$-form in a big and semi-positive class and $F$ is a smooth function with $\int_X e^F \omega_X^n = \int_X \theta^n$, then there exists a bounded $\theta$-psh function $u$  such that $(\ddc u+\theta)^n =e^F \omega_X^n$ and $\ddc u+ \theta$ is a K\"ahler form on an open Zariski dense subset in $X$ (see \cite{BEGZ} for details, also \cite{Yau1978,Kolodziej05}). This provides examples for elements in $\tilde{\mathcal{V}}(X,\omega_X,n,A,p_0,K)$. 

Observe that if $X'$ is a K\"ahler normal variety and $\pi: Y \to X'$ be a resolution of singularities and $\theta_Y$ is a K\"ahler form on $Y$, then the pull-back of currents in 
$$\mathcal{AK}(X',\theta_Y,n,p_0,A,K,\gamma)$$
 (see \cite[Definition 3.1]{Guo-Phong-Song-Sturm2} and also Definition \ref{def-AKmoi} below) belong to  $\tilde{\mathcal{V}}(Y,\theta_Y,n,A,p_0,K)$. Hence in order to treat metrics in $\mathcal{AK}(X',\theta_Y,n,p_0,A,K,\gamma)$ (which contains canonical metrics such as Ricci-flat or K\"ahler-Einstein ones if $X'$ has at worst log terminal singularities), it suffices to consider $\tilde{\mathcal{V}}(X,\omega_X,n,A,p_0,K)$. A uniform diameter bound and a uniform local non-collapsing estimate for volumes of metrics in $\mathcal{AK}(X',\theta_Y,n,p_0,A,K,\gamma)$ were established in \cite{Guo-Phong-Song-Sturm2}. The advantage of $\tilde{\mathcal{V}}(Y,\theta_Y,n,A,p_0,K)$ is that it does not require that the volume $\omega^n$ of $\omega$ has log type analytic singularities as in the case for $\mathcal{AK}(X',\theta_Y,n,p_0,A,K,\gamma)$ (see the condition \cite[Definition 3.1 (4)]{Guo-Phong-Song-Sturm2}, and Definition \ref{def-AKmoi} below).

For $\omega \in \tilde{\mathcal{V}}(X,\omega_X,n,A,p_0,K)$, let $U_\omega \subset X$ be an open Zariski dense subset on which $\omega$ is a smooth K\"ahler form. Let $(\widehat X, \omega)$ be the completion of the metric space $(U_\omega, \omega)$. We extends the smooth measure $\omega^n$ on $U_\omega$ to $\widehat X$ by putting $\omega^n(\widehat X \backslash U_\omega):=0$. One observes that $(\widehat X, \omega, \omega^n)$ is, modulo isometries, independent of the choice of $U_\omega$. We denote by $diam(\widehat X, \omega)$ the diameter of $(\widehat X, \omega)$.
Here is our first main result.

\begin{theorem}\label{the-main-univolume}  Let $A,K$ be positive constants and let $p_0>n$ be a constant. Let 
$$\tilde{\mathcal{V}}(X,\omega_X,n,A,p_0,K)$$
 be the family of K\"ahler metrics as above. Then there exist constants $C>0, q>0$ such that
 $$diam(\widehat X, \omega) \le C, \quad \frac{vol_\omega(B_\omega(x,r))}{r^q V_\omega} \ge C^{-1}$$
  for every $\omega \in \tilde{\mathcal{V}}(X,\omega_X,n,A,p_0,K)$, every $r \in (0, diam(\widehat X, \omega)]$ and every $x \in \widehat X$, where $B_\omega(x,r)$ denotes the ball of radius $r$ centered at $x \in \widehat X$ in $(\widehat X, \omega)$, and $vol_\omega(B_\omega(x,r))$ denotes the volume  of $B_\omega(x,r)$ with respect to $\omega^n$. 
  \end{theorem}

Shortly after the first version of this paper appeared in ArXiv, similar estimates as in Theorem \ref{the-main-univolume} for $\mathcal{V}(X,\omega_X,n,A,p_0,K)$ were obtained independently in \cite{GuedjTo-diameter,GPSS_bodieukien} by improving a key lemma in \cite{Guo-Phong-Song-Sturm} (using auxiliary complex Monge-Amp\`ere equations). 

We would like to point out that that since there is no assumption on the singularity of $\log (\omega^n/\omega_X^n)$ for $\omega$ in the family $\tilde{\mathcal{V}}(X,\omega_X,n,A,p_0,K)$, it is not clear if one can approximate such a metric smoothly on compact subsets in an open Zariski dense subset in $X$ by metrics in $\mathcal{V}(X,\omega_X,n,A,p_0,K)$. Such an approximation is crucial for the proof of Sobolev inequalities for metrics  in $\mathcal{AK}(X',\theta_Y,n,p_0,A,K,\gamma)$  in \cite{Guo-Phong-Song-Sturm2}.

Let $\mathcal{A}(X,\omega_X,n,A,p_0,K)$ be the set of all $(\widehat X, \omega)$ for $\omega \in \tilde{\mathcal{V}}(X,\omega_X,n,A,p_0,K)$. Observe that $(\widehat X, \omega)$ is compact because the diameter of $(U_\omega,\omega)$ is finite.
As in \cite[Theorem 1.2]{Guo-Phong-Song-Sturm}, by Theorem \ref{the-main-univolume}, we obtain immediately 
  
\begin{corollary} \label{cor-precompactGromov}   Let $A,K$ be positive constants and let $p_0>n$ be a constant. Then, the set $\mathcal{A}(X,\omega_X,n,A,p_0,K)$ is relatively compact in the Gromov-Hausdorff topology on the space of compact metric spaces.
\end{corollary}

We would like to stress that the feature of these families of metrics in this paper as well as in previous ones (e.g, \cite{Guo-Phong-Song-Sturm2,Guo-Phong-Song-Sturm,GGZ-logcontiu,Vu-log-diameter}) is that they present a large geometrically meaningful class of K\"ahler metrics whose metric properties can still be, to some extent, uniformly controllable.  Such things are remarkable and go beyond the convergence theory of Riemannian manifolds for which assumptions on uniform boundedness of curvatures are usually required.

We now give an application of Theorem \ref{the-main-univolume} to K\"ahler-Ricci flows. Let $\omega_0$ be a smooth K\"ahler metric on $X$. Consider the normalized K\"ahler-Ricci flow starting from $\omega_0$:
 $$\partial_t \omega(t)= - Ric(\omega(t)) - \omega(t), \quad \omega(0)=\omega_0.$$
 It is well-known that the flow exists for $t \in [0, \infty)$ if and only if $K_X$ is nef (see \cite{TianZhang}).  A main problem is to study the convergence of $(X, \omega(t))$ for $t \to \infty$. One can consult the very recent survey \cite{Tosatti-surveyKE} for an overview. It was proved in \cite{Guo-Phong-Song-Sturm} that if $K_X$ is nef and the Kodaira dimension of $X$ is nonnegative, then the diameter of $(X, \omega(t))$ is bounded uniformly in $t$; see also \cite{Jian-Song} for the case where $K_X$ is semi-ample. By Theorem \ref{the-main-univolume}, we can remove the assumption on the nonnegativity of the Kodaira dimension and obtain 

\begin{theorem} \label{the-KEflow} Assume that $K_X$ is nef. Then the diameter of $(X, \omega(t))$ is bounded uniformly in $t \in [0,\infty)$ and there exist constants $q>0, C>0$ such that 
$$\frac{vol_{\omega(t)}(B_{\omega(t)}(x,r))}{r^q V_{\omega(t)}} \ge C$$
 for every $x \in X$ and $r \in (0, diam(X, \omega(t)]$. In particular, there exists a sequence $(t_j)_{j \in \N} \subset (0, \infty)$ converging to $\infty$ so that the sequence of compact metric spaces $(X, \omega(t_j))$ converges, in the Gromov-Hausdorff topology, to a compact metric space. 
\end{theorem}

We now consider the singular setting. We introduce the following family of singular metrics. Let $X'$ be an n-dimensional compact normal K\"ahler space. Let $\pi: Y \to X'$ be a resolution of singularities obtained by successive blow ups along smooth submanifolds as in \cite{Bierstone_Milman} and let $\theta_Y$ be a smooth K\"ahler metric  on $Y$. We recall that $Y$ is automatically K\"ahler by \cite[Lemma 2.2]{Coman-Ma-Marinescu-normal}.  Let $\eta$ be a closed positive $(1,1)$-current on $X'$ admitting local potentials. Then we define $\pi^* \eta$ as usual by pulling back local potentials of $\eta$ under $\pi$. Thus $\pi^* \eta$ is a closed positive $(1,1)$-current on $Y$. We say that $\eta$ has \emph{bounded potentials} if there are  an open covering $V_1, \ldots, V_m$ of $X'$ and bounded psh functions $\varphi_j$ on $V_j$ for $1 \le j \le m$ such that $\eta= \ddc \varphi_j$ on $V_j$ for every $j$.

\begin{definition} \label{def-AKmoi} Let $(X',\omega_{X'})$ be an n-dimensional compact normal K\"ahler space.  For positive  constants $A,K$, and $p_0>n$, we define $\tilde{\mathcal{V}}(X',\omega_{X'},n,A,p_0,K)$ to be the space of closed positive $(1,1)$-currents $\eta$ on $X'$ having local potentials such that the following properties hold:

(i) $\eta$ is a smooth K\"ahler form on an open Zariski dense subset $U_\eta$ in $X'$ and has bounded potentials,

(ii) $\int_{X'} \eta\wedge \omega^{n-1}_{X'} \le A$,

(iii) $V_\eta^{-1}\eta^n = f \omega^n_{X'}$ for some $f \ge 0$ with $\int_{X'} f |\log f|^{p_0} \omega^n_{X'} \le K$.
\end{definition}

Observe that if $X'=X$ is smooth, then $\tilde{\mathcal{V}}(X',\omega_{X'},n,A,p_0,K)$ is exactly equal to $\tilde{\mathcal{V}}(X,\omega_X,n,A,p_0,K)$ treated above.  
For every $\eta \in \tilde{\mathcal{V}}(X',\omega'_X,n,A,p_0,K)$, we define the completion $(\widehat X', \eta, \eta^n)$ as in the smooth setting. We underline that  $\tilde{\mathcal{V}}(X',\omega_{X'},n,A,p_0,K)$ is intrinsic, thus is more convenient to use than the class of metrics considered in \cite[Definition 3.1]{Guo-Phong-Song-Sturm2}.  
The following is our second main result. 

\begin{corollary} \label{cor-singularsetting} Let $A,K$ be positive constants and let $p_0>n$ be a constant. Let $X'$ be an n-dimensional compact normal K\"ahler space.  Then there exist constants $C>0, q>0$ such that
 $$diam(\widehat X', \eta) \le C, \quad \frac{vol_{\eta}(B_\eta(x,r))}{r^q V_\eta} \ge C^{-1} $$
  for every $\eta \in \tilde{\mathcal{V}}(X',\omega_{X'},n,A,p_0,K)$ and every $r \in (0, diam(\widehat X', \eta)]$ and every $x \in \widehat X'$, where $B_\eta(x,r)$ denotes the ball of radius $r$ centered at $x \in \widehat X'$ in $(\widehat X', \eta)$. 
\end{corollary}

To explain our third main result, we observe that by fixing a basis $\alpha_1, \ldots, \alpha_{n_0}$ of $H^{1,1}(X,\R)$ (where $n_0:= \dim H^{1,1}(X,\R)$) we can identify $H^{1,1}(X,\R)$ with $\R^{n_0}$ via an obvious isomorphism. We endow $H^{1,1}(X,\R)$ with the Euclidean norm $\| \cdot \|$ on $\R^{n_0}$. Hence for each $\alpha \in H^{1,1}(X,\R)$ we get a smooth closed form $\theta_\alpha \in \alpha$ with $\|\theta_\alpha\|_{\mathcal{C}^k} \le C_k \|\alpha\|$, for every $k \in \N$,  where the constant $C_k$ depends only on $k, n_0$. Let $A_1>0$ be a constant and let  $\mathcal{M}(X,A_1)$ be the set of $\omega \in \mathcal{K}(X)$ such that  $\omega= \theta_\alpha+ \ddc \varphi_\omega$ for some $\theta_\alpha$-psh function with  $\|\varphi_\omega\|_{L^\infty(X)} \le A_1$ and $\|\alpha\| \le A_1$, where $\alpha$ is the cohomology class of $\omega$.  It is well-known that for every $A>0,K>0$, $p_0>n$,  there exists a constant $A_1$ such that 
$$\mathcal{V}(X,\omega_X,n,A,p_0,K) \subset \mathcal{M}(X,A_1).$$
It is difficult to expect that one can obtain similar geometric estimates for metrics in $\mathcal{M}(X,A_1)$. However, motivated by \cite[Theorem 1.1]{Guo-Song-localnoncollapsing} (which can be viewed as an analogue of Perelman's $\kappa$-noncollapsing theorem), we introduce the following class: for $x_0 \in X, R_0\in (0,3]$ (one can replace $3$ by any finite positive number), we  define $\mathcal{M}(X,A_1,x_0, R_0,p_0,K)$ to be the subclass of $\mathcal{M}(X,A_1)$ consisting of $\omega$ so that $V^{-1}_\omega\omega^n= f \omega_X^n$ with $\|f\|_{L^1\log L^{p_0}(B_\omega(x_0, R_0))} \le K$. The aim here is that we will try to control the metric properties of $\omega$ when $\omega^n$ has only fine regularity on the geodesic ball $B_\omega(x_0, R_0)$.  Observe that
$$\mathcal{V}(X,\omega_X,n,A,p_0,K) \subset \mathcal{M}(X,A_1,x_0, R_0,p_0,K).$$ 
Here is our third main result. 

\begin{theorem} \label{the-volumelocal} Let $A_1,x_0,R_0,p_0$ be as above. Then there exist constants $q>0, C_1>0$ independent of $x_0,R_0$ (but depending on $A_1,p_0,\omega_X$) so that 
$$\frac{vol_\omega(B_\omega(x_0,r))}{r^q V_\omega} \ge C_1$$
for every $\omega \in \mathcal{M}(X,A_1,x_0, R_0,p_0,K)$ and $r\in (0, R_0/2]$.
\end{theorem}

A closer look at the proof of \cite[Theorem 1.1]{Guo-Song-localnoncollapsing} shows that the arguments in \cite{Guo-Song-localnoncollapsing} provides an analogous lower bound as in Theorem \ref{the-volumelocal} if we assume additionally that the Ricci curvature of $\omega$ on $B_\omega(x_0,R_0)$ is bounded from below uniformly. 
We can in fact enlarge $\mathcal{M}(X,A_1,x_0, R_0,p_0,K)$ a bit more to contain $\tilde{\mathcal{V}}(X,\omega_X,n,A,p_0,K)$ and Theorem \ref{the-volumelocal} still hold. But to simplify the presentation, we only state the above version of Theorem \ref{the-volumelocal}.  We obtain the following consequence.

\begin{corollary} \label{cor-volumelocal} Let $A_1,p_0,K$ be as above. Let $U$ be an open subset in $X$. Let $$\mathcal{M}(X,A_1,p_0,K,U)$$ be the subclass of $\mathcal{M}(X,A_1)$ consisting of $\omega$ so that $V^{-1}_\omega \omega^n= f \omega_X^n$ with $\|f\|_{L^1\log L^{p_0}(U)} \le K$.  Then there exist  constants $q>0,$ $C>0$ so that 
$$ \dist_\omega(x, \partial U) \le C, \quad \frac{vol_\omega(B_\omega(x,r))}{r^q V_\omega} \ge C^{-1}$$
for every $\omega \in \mathcal{M}(X,A_1,p_0,K,U)$, $x \in U$, $r\in (0, \frac{1}{2}\dist_\omega(x, \partial U)]$, where 
$\dist_\omega(x, \partial U)$ denotes the distance from $x$ to $\partial U$ in $X$.
\end{corollary}


We now comment on proofs of our main results. If one admits \cite{GuedjTo-diameter,GPSS_bodieukien}, then the proof of Theorem \ref{the-main-univolume} is quite short, and follows from a convergence property of Monge-Amp\`ere operators of bounded potentials, see Lemma \ref{le-hoitulientuc} below. We however present a proof of Theorem \ref{the-main-univolume} using only results from \cite{Guo-Phong-Song-Sturm,Guo-Phong-Song-Sturm2}. The advantage is that the proof of Theorem \ref{the-volumelocal} is done in a similar manner. The strategy consists of two ingredients. The first one is the fundamental Sobolev inequality proved recently in \cite{Guo-Phong-Song-Sturm2}. The second one is a crucial energy estimate (Proposition \ref{pro-sobo-etatoT} below) for functions in the Dinh-Sibony-Sobolev space. In order to obtain such an estimate, we will develop further the calculus on that space initiated in \cite{DS_decay,Dinh-Ma-Marinescu,Vigny-Vu-Lebesgue}, especially, we establish an integration by parts formula; see Theorem \ref{th-integrationbyparts} below. \\

\noindent
\textbf{Acknowledgments.} We thank Bin Guo for numerous fruitful discussions. Thanks are also due to Duc-Bao Nguyen and Valentino Tosatti for their comments on the first draft of the paper. This work is partially supported by the Deutsche Forschungsgemeinschaft (DFG, German Research Foundation)-Projektnummer 500055552 and by the ANR-DFG grant QuaSiDy, grant no ANR-21-CE40-0016.

\section{Dinh-Sibony-Sobolev space}

Recall that $d=\partial + \bar{\partial}$ and $d^c:=\frac{i}{2 \pi} (\bar \partial -\partial)$. Let $U$ be a bounded open subset in $\C^n$ and let $\omega_{\C^n}$ be the standard K\"ahler form on $\C^n$.
Let $W^{1,2}(U)$ be the Sobolev space of function $u$ in $L^2(U)$ such that $\nabla u \in L^2(U)$. Let $W^{1,2}_*(U)$ be the subspace of $W^{1,2}(U)$ consisting of $u \in  W^{1,2}(U)$ such that  there exists a positive closed current $T$ of bidegree $(1,1)$ and of finite mass (i.e., $\int_U T \wedge \omega_{\C^n}^{n-1} < \infty$) on $U$ such that  
\begin{equation}\label{bound_by_currentlocal}
d u \wedge d^c u \leq T.
\end{equation}
We note that $du \wedge \dc u = \frac{i}{\pi} \partial u \wedge \bar \partial u$.  When $n=1$, then $W^{1,2}_*(U)=W^{1,2}(U)$ since $d u \wedge d^c u$  is already a positive measure. The space $W^{1,2}_*(U)$ is called \emph{Dinh-Sibony-Sobolev space}. It was also known under the name ``complex Sobolev space".   
The space $W^{1,2}_*(U)$ was introduced in \cite{DS_decay}, see also \cite{Vigny} for further properties  and \cite{DoNguyenW12} for a higher version of this space. We refer to \cite{DLW,DLW2,Vigny_expo-decay-birational,Vu_nonkahler_topo_degree} for applications of this space to dynamics; see also \cite{DinhMarinescuVu,DKC_Holder-Sobolev} for recent applications in Monge-Amp\`ere equations.

The space $W^{1,2}_*(U)$ is a Banach space endowed with the norm
\[ \|u\|_*^2= \int_U |u|^2 \omega_{\C^n}^{n} +  \inf \big\{ \int_U T \wedge \omega_{\C^n}^{n-1}\big\} \]    
where the infimum is taken over all closed positive currents $T$ of bidegree $(1,1)$ satisfying \eqref{bound_by_currentlocal}; see \cite[Proposition 1]{Vigny}. Observe that $\|u\|_* \ge c_n \|u\|_{W^{1,2}(U)}$, where the latter is the usual norm on $W^{1,2}(U)$, and $c_n>0$ is a dimensional constant. When $U$ is a (connected) bounded open set with $\mathcal{C}^1$ boundary, one can use the following variant of $*$-norm:
$$\|u\|'_*:= \|u\|_{L^1(U)}+ \inf \big\{ \bigg(\int_U T \wedge \omega_{\C^n}^{n-1}\bigg)^{1/2}\big\}$$
as in \cite{Vigny-Vu-Lebesgue}. In this case, by Poincar\'e's inequality, the norms $\|\cdot\|'_*$ and $\|\cdot\|_*$ are equivalent. Working with $U$ having $\mathcal{C}^1$ boundary puts no restriction in practice as long as results in consideration are locally in $U$ (as it it the case in \cite{DinhMarinescuVu,Vigny-Vu-Lebesgue} and in this paper).

We recall that if $u,v \in W^{1,2}_*(U)$, then $\max\{u,v\}$ and $\min\{u,v\}$ are also in $W^{1,2}_*(U)$; in particular, $\max\{u,0\}, \min\{u,0\}$, $|u|$ are in $W^{1,2}_*(U)$ (see \cite[Proposition 4.1]{DS_decay}). A function $u$ in $W^{1,2}_*(U)$ is a priori not well-defined pointwise. However, there are good representatives for $u$ defined as follows (see \cite{DinhMarinescuVu} and \cite{Vigny-Vu-Lebesgue}).

Let $u \in W^{1,2}_*(U)$.   
Fix a smooth radial cut-off  function $\chi$ on $\C^n$ such that $\chi$ is compactly supported, $0 \le \chi \le 1$, and $\int_{\C^n} \chi \omega^n =1$. Let 
 $$u_\epsilon(z):= \epsilon^{-2n}\int_{\C^n} u(z-x) \chi(x/\epsilon) \omega^n,$$
for $\epsilon \in (0,1]$. We call $u_\epsilon$ \emph{the standard regularisation} of $u$. For every open subset $U' \Subset U$, we have $u_\epsilon \to u$ in $W^{1,2}(U')$.  A Borel function $u': U \to \R$ is said to be \emph{a good representative of $u$} if $u'= u$ almost everywhere and if for every standard regularisation $(u_\epsilon)_\epsilon$, we have that  $u_\epsilon$ converges pointwise to $u'$ as $\epsilon \to 0$ outside some pluripolar set in $U$. Good representatives alway exist and differ from each other by a pluripolar set (see \cite[Theorem 1.1]{Vigny-Vu-Lebesgue} and \cite[Theorem 2.10]{DinhMarinescuVu}).

Let $v_1, \ldots, v_n$ be bounded psh functions on $\Omega$ and define $\mu:= \ddc v_1 \wedge \cdots \wedge \ddc v_n$. Such a measure $\mu$ is called a Monge-Amp\`ere measure of bounded potentials.  Note that $\mu$ is a positive measure  having no mass on  pluripolar sets because the capacity of every pluripolar set is zero. Using good representations, we can integrate any nonnegative $u\in W^{1,2}_*(U)$ against $\mu$ by putting $\langle \mu, u\rangle:= \langle \mu, u' \rangle$, where $u'$ is a good representative of $u$. One can see that this definition is independent of the choice of $u'$.  More generally, we can define in the same way $\langle \mu, \phi(u)\rangle$ for any positive Borel function $\phi$ defined everywhere on $\R.$  Every function in $W^{1,2}_*(U)$ is locally square integrable with respect to Monge-Amp\`ere measures of bounded potentials; see \cite[Proposition 2.13]{DinhMarinescuVu}.

There is a global version of $W^{1,2}_*(U)$ as follows. Let $(X, \omega_X)$ be a compact Hermitian manifold. Let $W^{1,2}(X)$ be the Sobolev space of Borel measurable functions $u$ on $X$ such that $u$ and $\nabla u$ are  square integrable functions (with respect to $\omega_X^n$).  Let $W^{1,2}_*(X)$ be the subset of $W^{1,2}(X)$ consisting of functions $u$ satisfying that  there exists a positive closed current $T$ of bidegree $(1,1)$ on $X$ such that  
\begin{equation}\label{bound_by_current}
d u \wedge d^c u \leq T.
\end{equation}
Observe that if $U$ is a local chart on $X$ such that the closure of $U$ is contained in a bigger local chart biholomorphic to a bounded open subset in $\C^n$, then $W^{1,2}_*(X) \subset W^{1,2}_*(U)$.
The space $W^{1,2}_*(X)$  enjoys similar properties as its local counterpart. First it is a Banach space endowed with the norm
\[ \|u\|_*^2= \int_X |u|^2 \omega_X^{n} +  \inf \big\{ \int_X T \wedge \omega_X^{n-1}\big\} \]    
where the infimum is taken over all closed positive currents $T$ of bidegree $(1,1)$ satisfying \eqref{bound_by_current}. Secondly, for $u,v \in W^{1,2}_*(X)$, we have that $\max\{u,v\}, \min\{u,v\}$ also belong to $W^{1,2}_*(X)$. The notion of good representatives extends naturally to $W^{1,2}_*(X)$, and every element in $W^{1,2}_*(X)$ possesses good representatives as in the  local setting (as a Borel function from $X$ to $\R$).

\subsection{Quasi-continuity}

Recall that for every Borel set $E$ in $U$, the relative capacity of $E$ (in $U$) is defined by  
$$\capK(E,U):= \sup \big\{\int_E (\ddc v)^n: 0 \le v \le 1, \, v \in \PSH(U) \big\},$$ 
where $\PSH(U)$ denotes the set of psh functions on $U$. 

\begin{proposition}\label{pro-quasiprojective} (\cite[Theorem 2.10]{DinhMarinescuVu})
Let $u \in W^{1,2}_*(U)$. Then for every $\epsilon>0$, there exists an open subset $U'$ in $U$ such that $u$ is continuous on $U \backslash U'$ and $\capK(U',U) \le \epsilon$.
\end{proposition}

\begin{lemma}\label{le-chuanLpcuauWsaodia} (\cite[Lemma 2.5]{Vigny-Vu-Lebesgue})  Let $C \ge 1,p \ge 1$ be constants.
Let $u \in W^{1,2}_*(U)$ with $\|u\|_* \le C$ such that there is a bounded psh function $0 \le \psi_1 \le C$ on $U$ satisfying $du \wedge \dc u \le \ddc \psi_1$. Let  $\psi_2$ be a psh function on $U$ with  $0 \le \psi_2 \le C$ and $\eta:= \ddc \psi_2$. Then, for every compact subset $E$ in $U$, there exists a constant $C'>0$ depending only on $p,E,C$ such that 
$$\int_E |u|^p \eta^n \le C'.$$
\end{lemma}

Consequently, we obtain

\begin{lemma}\label{le-hoitulientucdia} Let $\eta$ be a closed positive current of bounded potentials and $(\eta_k)_k$ be a sequence of closed positive $(1,1)$-currents such that 

(i) $\eta_k^n$ converges weakly to $\eta^n$ on $U$ as $k \to \infty$,

(ii) we can write $\eta_k= \ddc \psi_k$, where  $\|\psi_k\|_{L^\infty}$  are bounded uniformly in $k$.

Let $u \in W^{1,2}_*(U)$  such that $u \ge 0$ and there is a closed positive current $T$ of bounded potentials such that $du \wedge \dc u \le T$. Let $\chi$ be a smooth function compactly supported on $U$. Then for every constant $p \ge 0$, we have
$$\int_U \chi u^p \eta_k^n \to \int_U \chi u^p \eta^n$$
as $k \to \infty$.
\end{lemma}

\proof First we assume $\|u\|_{L^\infty} \le C_0$ for some constant $C_0$. By hypothesis, for every continuous function $f$, there holds 
$$\int_U \chi f \eta_k^n \to \int_U \chi f \eta^n.$$
Recall that $u$ is quasi-continuous, i.e, for every constant $\epsilon>0$, there exists an open subset $U'$ in $U$ such that $\capK(U',U) \le \epsilon$ and $u$ is continuous on $U \backslash U'$. 
Extend $u|_{U \backslash U'}$ to a continuous function $u'$ on $U$. We have thus
$$\int_U \chi u'^p \eta_k^n \to \int_U \chi u'^p \eta^n.$$
On the other hand, one has
\begin{align*}
\int_U \chi |u^p- u'^p| \eta_k^n &\le \int_{U'\cap \supp \chi} |u^p| \eta_k^n + \|u'\|^p_{L^\infty(\supp \chi)} \int_{U'} \eta_k^n \\
&\lesssim   \capK(U',U) \lesssim \epsilon.
\end{align*}
A similar estimate also holds for $\eta$ in place of $\eta_k$. It follows that 
$$\int_U \chi u^p \eta_k^n \to \int_U \chi u^p \eta^n$$
as desired if $u$ is bounded. 

Now let $M \ge 1$ be a constant and $0 \le u_M:= \min \{u, M\} \le M$. By the first part of the proof, one gets
$$\int_U \chi u_M^p \eta_k^n \to \int_U \chi u_M^p \eta^n.$$
Now compare 
$$\int_U \chi |u_M^p - u^p| \eta_k^n = \int_{\{u \ge M\}}\chi u^p \eta_k^n \le M^{-1}\int_U \chi u^{p+1} \eta_k^n \le C_2/ M$$
for some constant $C_2>0$ independent of $k$ by Lemma \ref{le-chuanLpcuauWsaodia}. A similar estimate holds for $\eta$ in place of $\eta_k$. The desired estimate thus follows by taking $M \to \infty$. 
\endproof

Let $(X,\omega_X)$ be a compact K\"ahler manifold of dimension $n$. By Lemmas \ref{le-chuanLpcuauWsaodia} and  \ref{le-hoitulientucdia}, we obtain immediately the following analogous estimates in the global setting.

\begin{lemma}\label{le-chuanLpcuauWsao}  Let $C \ge 1,p \ge 1$ be constants.
Let $u \in W^{1,2}_*(X)$ with $\|u\|_* \le C$ such that there is a closed positive current $T$ of bounded potentials such that $du \wedge \dc u \le T$, and $T= \theta_1+ \ddc\psi_1$ for some  smooth $\theta_1$ with $\|\theta_1\|_{\mathcal{C}^2} \le C$, and $0 \le \psi_1 \le C$. Let $\eta$ be a closed positive $(1,1)$-current on $X$ such that we can write $\eta= \theta_2+ \ddc \psi_2$, for some  smooth $\theta_2$ with $\|\theta_1\|_{\mathcal{C}^2} \le C$, and $0 \le \psi_2 \le C$. Then, there exists a constant $C'>0$ depending only on $p,X,C$ such that 
$$\int_X |u|^p \eta^n \le C'.$$
\end{lemma}

\proof   We cover $X$ by local charts $U$ such that $U$ is a relatively compact subset in a bigger local chart, and $U$ is biholomorphic to a ball in $\C^n$  on which we can write $T= \ddc v_1$, $\eta= \ddc v_2$ for psh functions $v_1,v_2$ on $U$ with $0 \le v_1, v_2 \le C_1$ for some constant $C_1>0$ depending only on $C,U$, here we have used hypothesis on $\theta_j, \psi_j$. Now the desired estimate follows from  Lemma \ref{le-chuanLpcuauWsaodia}. 
\endproof

\begin{lemma}\label{le-hoitulientuc} Let $\eta$ be a closed positive current of bounded potentials and $(\eta_k)_k$ be a sequence of closed positive $(1,1)$-currents  on $X$ such that 

(i) $\eta_k^n$ converges weakly to $\eta^n$ as $k \to \infty$,

(ii) we can write $\eta_k= \theta_k + \ddc \psi_k$, where $\theta_k$ is smooth and $\|\theta_k\|_{\mathcal{C}^2}, \|\psi_k\|_{L^\infty}$  are bounded uniformly in $k$.

Let $u \in W^{1,2}_*(X)$  such that $u \ge 0$ and there is a closed positive current $T$ of bounded potentials such that $du \wedge \dc u \le T$. Then for every constant $p \ge 0$, we have
$$\int_X u^p \eta_k^n \to \int_X u^p \eta^n$$
as $k \to \infty$.
\end{lemma}

\subsection{Integration by parts}

We develop now calculus with functions in $W^{1,2}_*(U)$, where $U$ is again a bounded open subset in $\C^n$.  
Let  $u_k \in W^{1,2}_*(U)$ for $k \in \N$ and $u \in W^{1,2}_*(U)$.
We say that $u_k \to u$ \emph{weakly} in $W^{1,2}_*$ if 
$u_k \to u$  in  the sense of distributions and $\|u_k\|_*$ is uniformly bounded.
By \cite{DinhMarinescuVu}, we say that $u_k \to u$ \emph{nicely}   
if $u_k \to u$ weakly in $W^{1,2}_*(U)$ and  for every $x \in U$, 
there exist an open neighbourhood $U_x$ of $x$ and a psh
function $\varphi_k$ on $U_x$ such that 
\begin{align*} 
d u_k \wedge \dc u_k \le \ddc \varphi_k
\end{align*}
 for every $k$ and  $\varphi_k$ decreases to some psh  function $\varphi$ on $U_x$. Observe that in this case, we also have $d u \wedge \dc u \le \ddc \varphi$ (see \cite[p. 251]{Vigny}).  By \cite[Corollary 2.11]{DinhMarinescuVu}, if $u_k \to u$ nicely, then $u_k \to u$ in capacity, where we have implicitly identified $u,u_k$ with their good representations. 
Let $(u_\epsilon)_\epsilon$ be standard regularisations of $u$.  For every open subset $U' \Subset U$, we have $u_\epsilon \to u$ in $W^{1,2}(U')$  and $u_\epsilon \to u$ nicely in $W^{1,2}_*(U')$. More precisely, if  $d u \wedge \dc u \le \ddc \varphi$, then $d u_\epsilon \wedge \dc u_\epsilon \le \ddc \varphi_\epsilon$, where $\varphi_\epsilon$ is the standard regularisation of $\varphi$ with the same cut off function as that in the definition of $u_\epsilon$ (see \cite[Lemma 5]{Vigny}).

The following fact will be very useful later.

\begin{proposition}\label{p:W12-MA-general}  (\cite[Proposition 2.13]{DinhMarinescuVu}) 
Let $K$ be a compact subset in $U$. Assume that $u_k \to u$ in $W^{1,2}_*(U)$ nicely. Then, we have
$$\lim_{k\to\infty} \sup_{v_1,\ldots, v_n}\int_K |u_k- u| \ddc v_1 \wedge \cdots  \wedge \ddc v_n =0,$$
where the supremum is taken over every psh function $v_1, \ldots, v_n$ on $U$ with $0 \le v_1,\ldots, v_n \le 1$. 
\end{proposition} 

\proof The statement of \cite[Proposition 2.13]{DinhMarinescuVu} is  weaker than what we claim here. But the proof of \cite[Proposition 2.13]{DinhMarinescuVu} actually gives the desired limit.
\endproof

We also recall the following version of the Cauchy-Schwarz inequality. 

\begin{lemma} \label{le-cauchyschsmooth} Let $R_1,\ldots, R_{n-1}$ be closed positive $(1,1)$-currents of bounded potentials on $U$ and let $f,g$ be bounded Borel functions with compact support in $U$. Let $u,v$ be smooth real functions. Then for $R:= R_1 \wedge \cdots \wedge R_{n-1}$, we have
$$\bigg|\int_U fg d u \wedge \dc v \wedge R \bigg| \le \bigg( \int_U |f|^2 du \wedge \dc u \wedge R\bigg)^{1/2}\bigg(\int_U |g|^2 dv \wedge \dc v \wedge R \bigg)^{1/2}$$
\end{lemma}

\proof
Let $\alpha:=f\partial u, \beta:= \overline g \partial v$ which are bounded $(1,0)$-forms. Observe that $fg du \wedge \dc v \wedge R=\frac{i}{\pi} \alpha \wedge \overline \beta \wedge R$, similar equalities also hold for $du \wedge \dc u$ and $dv \wedge \dc v$ instead of $du \wedge \dc v$. Thus the desired inequality follows from the standard Cauchy-Schwarz inequality applied to $\alpha, \beta$.
\endproof

Let $1 \le m \le n-1$ be an integer, and let  $R= R_1 \wedge \cdots \wedge R_m$, where $R_1,\ldots,R_m$ are closed positive $(1,1)$-currents of bounded potentials on $U$. Let $u \in W^{1,2}_*(U)$. We define $du \wedge R$ as follows. For every smooth form $\Phi$ of compact support in $U$, we put 
$$\langle d u \wedge R, \Phi \rangle:= -\int_U u d \Phi \wedge R.$$
One sees that $du \wedge R$ is a well-defined current on $U$. We define also $\partial u \wedge R$ to be the $(m+1,m)$-current such that $\langle \partial u \wedge R, \Phi\rangle = \langle d u \wedge R, \Phi \rangle$ for every smooth $(n-m-1,n-m)$-form $\Phi$. The current $\bar \partial u \wedge R$ is defined similarly. By a bi-degree reason, one gets $du \wedge R= \partial u \wedge R+ \bar \partial u \wedge R$.

\begin{lemma} \label{le-uknicelyuchico1duy} Let $(u_k)_k$ be a sequence in $W^{1,2}_*(U)$ converging nicely to $u$ in $W^{1,2}_*(U)$. Then, we have  
\begin{align} \label{ine-ukunicelychico1du}
d u_k \wedge R \to du \wedge R 
 \end{align}
 weakly as currents. Similar convergences hold for $\partial u, \bar \partial u$ instead of $d u$. In particular, the currents $du \wedge R, \partial u \wedge R, \bar \partial u \wedge R$ are of order $0$. 
 \end{lemma}
 
 \proof Let $\Phi$ be a  smooth $(2n-2m-1)$-form of compact support in $U$. Let $P_k:= d u_k \wedge R$, and $P:= du \wedge R$.  By definition, we have
$$\big|\langle P_k -P, \Phi \rangle\big| = \bigg| \int_U (u- u_k) d \Phi \wedge R\bigg| \le \|\Phi\|_{\mathcal{C}^1} \int_{\supp \Phi} |u-u_k| R \wedge \omega_{\C^n}^{n-m}$$ 
 which converges to $0$ by Proposition \ref{p:W12-MA-general}. Hence  (\ref{ine-ukunicelychico1du}) follows. Let $K$ be a compact subset in $U$ and let $U_1 \Subset U$ be an open subset containing $K$. Let $T$ be a closed positive $(1,1)$-current on $U$ such that $i\partial u \wedge \bar \partial u \le T$ and let $u_\epsilon$ be the standard regularisation of $u$. Since $u_\epsilon \to u$ nicely, and the problem is local, we can assume that $u_\epsilon$ is defined on $U$ and  $i \partial u_\epsilon \wedge \bar \partial u_\epsilon \le T_\epsilon$, where $T_\epsilon$ is a closed positive $(1,1)$-current of mass bounded uniformly in $U$. It follows that for every smooth function $v$  supported in $K$ with $i \partial v \wedge \bar \partial v \le \omega_{\C^n}$, one has 
\begin{align*} 
 \bigg| \int_U i\partial u_\epsilon \wedge \bar \partial v  \wedge R \wedge \omega_{\C^n}^{n-m-1} \bigg| &\le \bigg(\int_U i\partial u_\epsilon \wedge \bar \partial u_\epsilon  \wedge R \wedge \omega_{\C^n}^{n-m-1} \bigg)^{1/2} \times\\
 \quad &  \bigg(\int_U i\partial v \wedge \bar \partial v  \wedge R \wedge \omega_{\C^n}^{n-m-1} \bigg)^{1/2}\\
 &\le \bigg(\int_{\supp v} T_\epsilon  \wedge R \wedge \omega_{\C^n}^{n-m-1} \bigg)^{1/2} \bigg(\int_{\supp v} R \wedge \omega_{\C^n}^{n-m} \bigg)^{1/2}\\
 & \lesssim \|T_\epsilon\|_{U_1} \|R\|_{U_1} \lesssim 1 
 \end{align*}
 by Chern-Levine-Nirenberg inequality, where $\|R\|_{U_1}$ denotes the mass of $R$ on $U_1$, similarly for $\|T_\epsilon\|_{U_1}$. Letting $\epsilon \to 0$ gives
\begin{align}\label{ine-partialuv}
 \bigg| \int_U i\partial u_\epsilon \wedge \bar \partial v  \wedge R \wedge \omega_{\C^n}^{n-m-1} \bigg| \le C
 \end{align}
 for some constant $C$ independent of $v$. Let  $\Phi$ be a smooth $(n-m-1,n-m)$-form with support in $K$.  We can write $\Phi= \sum_{j=1}^n f_j \bar \partial z_j \wedge \Phi_j$, where $|f_j| \le C_1 \|\Phi\|_{\mathcal{C}^0}$ and $\Phi_j$ is a positive form such that $\Phi_j \le \omega_{\C^n}^{n-m}$, where $z=(z_1,\ldots, z_n)$ are coordinates on $\C^n$ and $C_1>0$ is a dimensional constant. This combined with (\ref{ine-partialuv}) implies that $\int_X \partial u_\epsilon \wedge R \wedge \Phi$ is bounded by a constant $C$ (independent of $\epsilon$) times $\|\Phi\|_{\mathcal{C}^0}$. Letting $\epsilon \to 0$ implies that $\partial u \wedge R$ is of order $0$. We argue similarly for $\bar \partial u \wedge R$. The proof is finished.
 \endproof

Let $\psi$ be  a bounded nonnegative psh function on $U$. Let  $R= R_1 \wedge \cdots \wedge R_{n-1}$, where $R_1,\ldots,R_{n-1}$ are closed positive $(1,1)$-currents of bounded potentials on $U$. Let $u \in W^{1,2}_*(U)$. We define $du \wedge \dc \psi \wedge R$ and $d \psi \wedge \dc u \wedge R$ as follows. It suffices to deal with $du \wedge \dc \psi \wedge R$ because we put formally $d \psi \wedge \dc u \wedge R:= du \wedge \dc \psi \wedge R$ (which is true for smooth $u, \psi$ by a bi-degree reason).

Let $\chi$ be a smooth function compactly supported on $U$. We define
$$\langle du \wedge \dc \psi \wedge R, \chi \rangle:= -\int_U u d \chi \wedge \dc \psi \wedge R-\int_U u\chi \ddc \psi \wedge R.$$
Both terms in the right-hand side are finite because $u$ is integrable with respect to every Monge-Amp\`ere measure of bounded potentials (recall here that we can always write $\chi= w_1-w_2$ as the difference of two bounded nonnegative psh functions, and $d w_j \wedge \dc \psi \wedge R= \frac{1}{2}\ddc(w_j+\psi)^2\wedge R- (w_j+\psi) \ddc(w_j+\psi) \wedge R$ for $j=1,2$). Thus we obtain a $(n,n)$-current $du \wedge \dc \psi \wedge R$. 

\begin{lemma} \label{le-uknicelyuchico1duy2} Let $(u_k)_k$ be a sequence in $W^{1,2}_*(U)$ converging nicely to $u$ in $W^{1,2}_*(U)$. Then   
\begin{align} \label{ine-ukunicelychico1du2}
d u_k \wedge \dc \psi \wedge R \to du \wedge \dc \psi \wedge R 
 \end{align}
 weakly as currents. Consequently, the current $du \wedge \dc \psi \wedge R$ is of order $0$ (hence is a signed measure), and for every compactly supported continuous functions $f,g$ on $U$, one obtains
\begin{align} \label{ine-CWine}
\bigg|\int_U fg du \wedge \dc \psi \wedge R\bigg| \le \bigg(\int_U |f|^2 T \wedge R\bigg)^{1/2}\bigg(\int_U |g|^2 d \psi \wedge\dc \psi \wedge R\bigg)^{1/2},
\end{align} 
where $T$ is a closed positive current on $U$ such that $du \wedge \dc u \le T$.
 \end{lemma}

\proof  $K \subset U$ be a compact subset and let $U_1 \Subset U$ be an open subset containing $K$.   Let $\chi$ be a smooth function  supported in $K$. Observe that
$$\int_U \chi d(u_k- u) \wedge \dc \psi \wedge R = -\int_U (u_k- u) d \chi \wedge \dc \psi \wedge R-\int_U (u_k- u) \chi \ddc \psi \wedge R$$
which converges to $0$ because of Proposition \ref{p:W12-MA-general}. Hence (\ref{ine-ukunicelychico1du2}) follows. Let $u_\epsilon$ be the standard regularisation of $u$. By Cauchy-Schwarz inequality (Lemma \ref{le-cauchyschsmooth}), one has (as in the proof of Lemma \ref{le-uknicelyuchico1duy})
\begin{align*}
\bigg|\int_U fg d u_\epsilon \wedge \dc \psi \wedge R \bigg| \le \bigg(\int_{U} |f|^2 d u_\epsilon \wedge \dc u_\epsilon \wedge R \bigg)^{1/2} \bigg(\int_U |g|^2 d \psi \wedge \dc \psi \wedge R \bigg)^{1/2}
\end{align*}
for every continuous functions $f,g$ with compact support on $U$.
Arguing now  as in the proof of Lemma \ref{le-uknicelyuchico1duy}, we obtain that 
$$\bigg|\int_U \chi d u_\epsilon \wedge \dc \psi \wedge R \bigg| \le C_{U_1} \|\chi\|_{L^\infty},$$
where  $C_{U_1}>0$ is a constant independent of $\epsilon$. Letting $\epsilon \to 0$ implies that $du \wedge \dc \psi \wedge R$ is of order $0$.  We check  (\ref{ine-CWine}) similarly.
\endproof

\begin{corollary} \label{cor-dudcpsinomass} Let $K \Subset U$ be a compact set.  There exists a constant $C>0$ such that for every $u \in W^{1,2}_*(U)$ with $\|u\|_* \le 1$ and every Borel set $E$ in $K$, there holds
$$\mu(E) \le C (\capK(E, U))^{1/2},$$
where $\mu$ is the total variation measure of $du \wedge \dc \psi \wedge R$. In particular, $\mu$ has no mass on pluripolar sets.
\end{corollary}

\proof Since $\|u\|_* \le 1$, there exists a closed positive $(1,1)$-current $T$ on $U$ with $\int_U T \wedge \omega_{\C^n}^{n-1}=1$ and $du \wedge \dc u \le T$.  By this and (\ref{ine-CWine}), we get
\begin{align*} 
\int_E d \mu  &\le \bigg(\int_E T \wedge R\bigg)^{1/2}\bigg(\int_E d \psi \wedge\dc \psi \wedge R\bigg)^{1/2}\\
&\le \bigg(\int_K T \wedge R\bigg)^{1/2}\bigg(\int_E d \psi \wedge\dc \psi \wedge R\bigg)^{1/2}\\
&\lesssim (\capK(E, U))^{1/2}
\end{align*} 
by Chern-Levine-Nirenberg inequality and the fact that $\psi$ is bounded.
\endproof


\begin{lemma} \label{le-uknicelyuchico1duy2them} Let $(u_k)_k$ be a sequence in $W^{1,2}_*(U)$ converging nicely to $u$ in $W^{1,2}_*(U)$. 
Let $v_1, \ldots, v_m$ be functions in $W^{1,2}_*(U)$ and let $v_{jk} \in W^{1,2}_*(U)$ converges nicely to $v_j$ as $k \to \infty$ for $1 \le j \le m$. Let $\rho$ be a smooth function on $\R^m$ with $\|\rho\|_{\mathcal{C}^1([-r_1,r_1] \times \cdots \times [-r_m, r_m])}< \infty$, where $r_j:= \|v_j\|_{L^\infty} \in [0, \infty]$ for $1 \le j \le m$. Then   
\begin{align} \label{ine-ukunicelychico1du23}
\rho(v_{1k},\ldots,v_{mk}) d u_k \wedge \dc \psi \wedge R \to \rho(v_1, \ldots, v_m) du \wedge \dc \psi \wedge R 
 \end{align}
 weakly as currents.
\end{lemma}

\begin{proof} By Corollary \ref{cor-dudcpsinomass}, both sides of (\ref{ine-ukunicelychico1du23}) make sense by using good representatives of $v_1,\ldots, v_m$. 
By Lemma \ref{le-uknicelyuchico1duy2}, we already have that $d u_k \wedge \dc \psi \wedge R$ converges weakly to  $du \wedge \dc \psi \wedge R$.
Using this, Corollary \ref{cor-dudcpsinomass} again,  quasi-continiuity of $v_1,\ldots,v_m$ and arguments similar to those in the proof of Lemma \ref{le-hoitulientucdia}, we obtain
\begin{align} \label{ine-ukunicelychico1du2yeu}
\rho(v_{1},\ldots,v_{m}) d u_k \wedge \dc \psi \wedge R \to \rho(v_1, \ldots, v_m) du \wedge \dc \psi \wedge R 
 \end{align}
 weakly as currents. Put 
 $$P_k:=\rho(v_{1k},\ldots,v_{mk}) d u_k \wedge \dc \psi \wedge R, \quad P:=\rho(v_{1},\ldots,v_{m}) d u_k \wedge \dc \psi \wedge R.$$
 We compare $P_k$ with $P$. To simplify the notations, we consider $m=1$. The general case is done using similar arguments. 
 
Let $\chi$ be a smooth function with compact support in $U$ and let $\delta>0$ be a positive constant. Let $Z:=[-r_1,r_1] \times \cdots \times [-r_m, r_m]$. By hypothesis,  $v_{1k}$ converges to $v_1$ in capacity.   By Corollary \ref{cor-dudcpsinomass}, we have
\begin{align*}
\bigg| \int_U \chi(P_k -P) \bigg| &\le \bigg| \int_{\{|v_{1k}- v_1| \ge \delta\}} \chi(P_k -P) \bigg|+ \bigg| \int_{\{|v_{1k}- v_1| \le \delta\}} \chi(P_k -P) \bigg|\\
&\lesssim \|\rho\|_{\mathcal{C}^0(Z)}\bigg(\capK\big(\{|v_{1k}- v_1| \ge \delta \} \cap \supp \chi, U \big)\bigg)^{1/2} + \delta \|\rho\|_{\mathcal{C}^1(Z)}\\
&= o_{k \to \infty}(1)+ \delta \|\rho\|_{\mathcal{C}^1(Z)}.
\end{align*}
Letting $k \to \infty$ and then $\delta \to 0$ gives the desired convergence.
\end{proof}

\begin{proposition}\label{pro-inter-onU}
Let $m\ge 1$ be an integer. Let $v_1,\ldots, v_m, u \in  W^{1,2}_*(U)$ and let $\psi$ be a bounded psh function. Let $R$ be the wedge product of $n-1$ closed positive $(1,1)$-currents of bounded potentials and let $\chi$ be a smooth function compactly supported on $U$. Let $\rho$ be a smooth function on $\R^m$ with $\|\rho\|_{\mathcal{C}^1([-r_1,r_1] \times \cdots \times [-r_m, r_m])}< \infty$, where $r_j:= \|v_j\|_{L^\infty} \in [0, \infty]$ for $1 \le j \le m$.  Assume that $u$ is bounded. Then we have
\begin{multline} \label{eq-tichphantungp}
\int_U \chi \rho(v_1,\ldots, v_m) du \wedge \dc \psi \wedge R= - \int_U u \rho(v_1,\ldots, v_m) d\chi \wedge \dc \psi \wedge R\\
 - \sum_{j=1}^m \int_U u \chi \partial_j \rho(v_1,\ldots, v_m) d v_j \wedge \dc \psi \wedge R- \int_U u \chi \rho(v_1,\ldots,v_m) \ddc \psi \wedge R,
 \end{multline}
 where $\partial_j \rho$ denotes the partial derivative of $\rho$ with respect to the $j$-th variable. 
\end{proposition}

\proof Let $v_{j\epsilon}$ be standard regularisations of $v_j$ for $1 \le j \le m$. Thus $v_{j\epsilon}$ converges nicely to $v_j$ as $\epsilon \to 0$, and $v_{j\epsilon}$ converges pointwise to $v_j$ outside some pluripolar sets. By this and Lebesgue's dominated convergence theorem (and the hypothesis on $\rho$), we get
$$\int_U \chi \rho(v_{1\epsilon},\ldots, v_{m\epsilon}) du \wedge \dc \psi \wedge R \to \int_U \chi \rho(v_1,\ldots, v_m) du \wedge \dc \psi \wedge R$$
as $\epsilon \to 0$. On the other hand, if we denote by $L(v_1,\ldots,v_m)$ the right-hand side of (\ref{eq-tichphantungp}), then  Lemma \ref{le-uknicelyuchico1duy2them} implies that 
$$L(v_{1\epsilon}, \ldots, v_{m \epsilon}) \to L(v_1, \ldots,v_m)$$
as $\epsilon \to 0$. Hence by considering $v_{j\epsilon}$ in place of $v_j$, we can assume that $v_j$'s are smooth. Now applying again Lemma \ref{le-uknicelyuchico1duy2them} shows that it suffices to treat the case where $u$ is smooth (we use the standard regularisation of $u$ here).  In this case, this is just the usual integration by parts for bounded psh functions. The proof is finished.  
\endproof

We note that the above definition of $du \wedge \dc \psi \wedge R$ extends obviously to the case where $\psi$ is a bounded quasi-psh function: in this case, we write locally $\psi= \psi_1+g$, where $\psi_1$ is psh and $g$ is a smooth function, and put 
$$du \wedge \dc \psi \wedge R:= du \wedge \dc \psi_1 \wedge R + du \wedge \dc g \wedge R$$
which is independent of choices of $\psi_1,g$ because the continuity of $du \wedge \dc \psi_1 \wedge R,$ $du \wedge \dc g \wedge R$ under sequences of standard regularisations of $u$. 
 
Consider now a compact complex manifold $X$. Let $\psi$ be a bounded quasi-psh function on $X$, $u \in W^{1,2}_*(X)$ and let $R$ be the wedge product of $n-1$ closed positive $(1,1)$-currents of bounded potentials. Observe that for every open subset $U$ of $X$ such that $\overline U$ is contained in a local chart of $X$, we have $u \in W^{1,2}_*(U)$. Thus, the local construction of $du \wedge \dc \psi \wedge R$ on local charts of $X$ gives a global well-defined signed measure on $X$. The following is a direct consequence of (\ref{ine-CWine}).

\begin{lemma} \label{le-cauchyschwarglobal} Let $u\in W^{1,2}_*(X)$ and let $f,g$ be bounded Borel functions on $X$. Let $T$ be a closed positive current on $X$ such that $du \wedge \dc u \le T$. Then the measure $du \wedge \dc \psi \wedge R$ is of bounded total variation on $X$, and there holds
$$\bigg|\int_X fg du \wedge \dc \psi \wedge R \bigg| \le \bigg(\int_X |f|^2 T \wedge R\bigg)^{1/2}\bigg(\int_X |g|^2 d \psi \wedge\dc \psi \wedge R\bigg)^{1/2}.$$ 
\end{lemma}  

\proof The fact that the measure is of bounded total variation follows from the compactness of $X$ and (\ref{ine-CWine}). The desired inequality is also deduced from these facts and approximating $f,g$ by continuous functions (Lusin's theorem).
\endproof

\begin{theorem} \label{th-integrationbyparts} Let $X$ be a compact complex manifold.
Let $m\ge 1$ be an integer. Let $v_1,\ldots, v_m, u \in W^{1,2}_*(X)$ and let $\psi$ be a bounded quasi-psh function. Let $R$ be the wedge product of $n-1$ closed positive $(1,1)$-currents of bounded potentials. Let $\rho$ be a smooth function on $\R^m$ with $\|\rho\|_{\mathcal{C}^1([-r_1,r_1] \times \cdots \times [-r_m, r_m])}< \infty$, where $r_j:= \|v_j\|_{L^\infty} \in [0, \infty]$ for $1 \le j \le m$. Assume that $u$ is bounded.  Then we have
\begin{multline*}
\int_X \rho(v_1,\ldots, v_m) du \wedge \dc \psi \wedge R=\\
 - \sum_{j=1}^m \int_X u \partial_j \rho(v_1,\ldots, v_m) d v_j \wedge \dc \psi \wedge R- \int_X u \rho(v_1,\ldots, v_m) \ddc \psi \wedge R,
 \end{multline*}
 where $\partial_j \rho$ denotes the partial derivative of $\rho$ with respect to the $j$-th variable. 
\end{theorem}

\proof
Let $(U_k)_{1 \le k \le l}$ be a covering of $X$ by local charts. Let $(\chi_k)_{1\le k \le l}$ be  a partition of unity subordinated to $(U_k)_{1 \le k \le l}$. We have 
\begin{align*}
\int_X \rho(v_1,\ldots, v_m) du \wedge \dc \psi \wedge R= \sum_{k=1}^l \int_{U_k} \chi_k \rho(v_1,\ldots, v_m) du \wedge \dc \psi \wedge R
\end{align*}
which is, by  Proposition \ref{pro-inter-onU}, equal to 
\begin{multline*}
-\sum_{k=1}^l   \int_{U_k} u \rho(v_1,\ldots, v_m) d\chi_k \wedge \dc \psi \wedge R -\\  \sum_{k=1}^l\int_{U_k} u \chi_k \sum_{j=1}^m \partial_j \rho(v_1,\ldots, v_m) d v_j \wedge \dc \psi \wedge R-\sum_{k=1}^l\int_{U_k} u \chi_k  \rho(v_1,\ldots, v_m) \ddc \psi \wedge R.
\end{multline*}
The first sum (in $k$) in the above expression is equal to 
$$\int_X u \rho(v_1,\ldots, v_m) d\big(\sum_{k=1}^l \chi_k \big)\wedge \dc \psi \wedge R=0$$
because $\sum_{k=1}^l \chi_k =1$, and $\supp \chi_k \subset U_k$.  We have also an analogous observation for the second and third sums. Hence we obtain the desired inequality. 
\endproof



\subsection{Energy estimates}

Recall that a dsh function is the difference of two quasi-psh functions. This notion was introduced in \cite{DS_allurepolynom}. A typical example of such functions is as follows. For two closed positive $(1,1)$-currents $R_1,R_2$ of the same cohomology class, we can write $R_1-R_2= \ddc \psi$ for some dsh function $\psi$.  Here is our key inequality.

\begin{proposition}\label{pro-sobo-etatoT} Let $M \ge 1$ and $p \ge 0$ be constants. Let $u \ge 1$  be an element in $W^{1,2}_*(X)$ such that $du \wedge \dc u \le T= \eta+ \ddc \psi$ with $0 \le \psi \le M$, where $\eta,T$ are closed positive $(1,1)$-current of bounded potentials, and $\psi$ is a dsh function. Then, for every $0 \le k \le n$, there holds  
$$\int_X u^p T^k \wedge \eta^{n-k} \le 2 M (p+1)^2 \int_X u^p T^{k+1} \wedge \eta^{n-k-1}.$$
Consequently,
$$\int_X u^p  T^k \wedge \eta^{n-k} \le 2^{n-k} (p+1)^{2(n-k)} M^{n-k} \int_X u^p T^n.$$
\end{proposition}

\proof The second desired inequality follows from the first one by induction. We prove now the first desired estimate.  It suffices to consider the case where $u$ is bounded. The general case follows by considering $\min\{u,c\}$ instead of $u$ ($c>1$ is a constant) and letting $c \to \infty$ (we use here Lemma \ref{le-chuanLpcuauWsao} and Lebesgue's dominated convergence theorem). Put $R:= T^k \wedge \eta^{n-k-1}$. The desired inequality reads as
$$\int_X u^p \eta \wedge R \le 2 M (p+1)^2 \int_X u^p T \wedge R.$$
We recall that $\psi$ itself is an element in $W^{1,2}_*(X)$ because $\psi$ is the difference of two bounded quasi-psh functions. Let 
$$J:= \int_X u^p  d \psi \wedge \dc \psi \wedge R.$$
By Theorem \ref{th-integrationbyparts} and the hypothesis that potentials of $T,\eta$ are bounded, we obtain
\begin{align*}
J = - p \int_X u^{p-1} \psi du \wedge \dc \psi \wedge R - \int_X u^p \psi \ddc \psi \wedge R.
\end{align*}
Let $I_1, I_2$ denote the first and second terms in the right-hand side of the last inequality respectively. Writing 
$$u^{p-1} d u \wedge \dc \psi= \big( u^{p/2-1} d u \big) \wedge \big(  u^{p/2} \dc \psi\big)$$
 and using Cauchy-Schwarz inequality (Lemma \ref{le-cauchyschwarglobal}), we get
\begin{align*}
I_1 & \le  p \bigg( \int_X  \psi u^{p-2} T \wedge R \bigg)^{1/2} \bigg(\int_X  \psi u^{p} d \psi \wedge \dc \psi \wedge R\bigg)^{1/2}\\
& \le  M p \bigg( \int_X u^{p-2} T \wedge R \bigg)^{1/2} \bigg(\int_X u^{p} d \psi \wedge \dc \psi \wedge R\bigg)^{1/2}\\
& \le   M p \bigg( \int_X u^{p} T \wedge R \bigg)^{1/2} \bigg(\int_X u^{p} d \psi \wedge \dc \psi \wedge R\bigg)^{1/2}
\end{align*}
because $u\ge 1$.  On the other hand,
\begin{align*}
I_2 &= -\int_X u^p \psi (\ddc \psi+ \eta) \wedge R+ \int_X \psi u^p \eta \wedge R\\
& \le M\int_X u^p \eta \wedge R.
\end{align*}
Observe
\begin{align*}
\int_X u^p \eta \wedge R  & = \int_X u^p T \wedge R- \int_X u^p \ddc \psi \wedge R\\
& \le \int_X u^p T \wedge R+ p \int_X u^{p-1} d u \wedge \dc \psi \wedge R\\
& \le \int_X u^p T \wedge R+ p \int_X u^{p-1} d u \wedge \dc \psi \wedge R\\ 
& \le \int_X u^p T \wedge R+  p \bigg( \int_X u^{p-2} T \wedge R \bigg)^{1/2} \bigg(\int_X u^{p} d \psi \wedge \dc \psi \wedge R\bigg)^{1/2}\\
& \le \int_X u^p T \wedge R+  p \bigg( \int_X u^{p} T \wedge R \bigg)^{1/2} \bigg(\int_X u^{p} d \psi \wedge \dc \psi \wedge R\bigg)^{1/2}
\end{align*}
by the above estimates for $I_1$ and the fact that $u \ge 1$. Hence,  we get 
$$I_2 \le M\int_X u^p T \wedge R+  M p \bigg( \int_X u^{p} T \wedge R \bigg)^{1/2} \bigg(\int_X u^{p} d \psi \wedge \dc \psi \wedge R\bigg)^{1/2}$$
It follows that
$$J \le M \int_X u^p T \wedge R+ 2M p \bigg( \int_X u^{p} T \wedge R \bigg)^{1/2} J^{1/2}.$$
Hence we get 
\begin{align}\label{ine-Lpdpsidpsi2phay}
J \le 4M^2(p+1)^2 \int_X u^p T \wedge R.
\end{align}
Now consider
\begin{align*} 
\int_X u^p \eta \wedge R  &= \int_X u^p T \wedge R- \int_X u^p \ddc \psi \wedge R \\
&= \int_X u^p T \wedge R + p \int_X u^{p-1}  du \wedge \dc \psi \wedge R \\
&\le  \int_X u^p T \wedge R + p \big(\int_X u^{p-2}  T \wedge R \big)^{1/2} \big(\int_X u^{p} d\psi \wedge \dc \psi \wedge R \big)^{1/2}\\
& \le \int_X u^p T \wedge R + p \big(\int_X u^{p}  T \wedge R \big)^{1/2} \big(\int_X u^{p} d\psi \wedge \dc \psi \wedge R \big)^{1/2}\\
& \le  (1+ 2 M (p+1)p)  \int_X u^p T \wedge R\\
& \le 2 M(p+1)^2 \int_X u^p T \wedge R
\end{align*}
by (\ref{ine-Lpdpsidpsi2phay}). This finishes the proof.
\endproof

We also need the following modified version of Proposition \ref{pro-sobo-etatoT}, in which $u$ is allowed to be close to $0$. 

\begin{proposition}\label{pro-sobo-etatoTsuadi} Let $M \ge 1$ and $p \ge 2n$ be constants. Let $u \ge 0$  be an element in $W^{1,2}_*(X)$ such that $du \wedge \dc u \le T= \eta+ \ddc \psi$ with $0 \le \psi \le M$, where $\eta,T$ are closed positive $(1,1)$-current of bounded potentials. Then, for every $0 \le k \le n$, there holds  
$$\int_X u^p T^k \wedge \eta^{n-k} \le 2 M (p+1)^2 \bigg(\int_X u^p T^{k+1} \wedge \eta^{n-k-1}+\int_X u^{p-2} T^{k+1} \wedge \eta^{n-k-1} \bigg).$$
Consequently,
$$\int_X u^p T^k \wedge \eta^{n-k} \le 2^{2(n-k)} (p+1)^{2(n-k)} M^{n-k} \int_X (u^p+ u^{p-2n}) T^n.$$
\end{proposition}

\proof The second desired inequality follows from the first one by induction and the observation that $u^s \le u^p+ u^{p-2n}$ for every $p-2n \le s \le p$ (here we used the hypothesis that $p-2n\ge 0$). We prove now the first desired estimate. The proof is essentially the same as that of Proposition \ref{pro-sobo-etatoT}. One just has to note that since we no longer have $u \ge 1$, we don't have $u^{p-2} \le u^p$.  Let the notations be as in the proof of Proposition \ref{pro-sobo-etatoT}. Put 
$$Q:=\int_X u^{p-2} T \wedge R+\int_X u^{p} T \wedge R.$$
By the proof of Proposition \ref{pro-sobo-etatoT}, we see that 
$$I_1 \le Mp Q^{1/2} J^{1/2}, \quad I_2 \le M Q+ M p Q^{1/2} J^{1/2}.$$
Thus
$$J \le M Q+ 2Mp Q^{1/2} J^{1/2}.$$
We infer that 
$$J \le 4M^2(p+1)^2 Q.$$
The rest is now done as in the proof of Proposition \ref{pro-sobo-etatoT}. 
\endproof

\section{Proof of main results}

 We first prove Theorem \ref{the-main-univolume}. We recall the following standard result.

\begin{lemma} \label{le-Moseriteration} Let $C>0,\lambda \ge 1, \beta>1$ be constants.  Let $\mu$ be a probability measure and let $u$ be a $\mu$-measurable function such that for every constant $p \ge \lambda$, we have 
$$\|u\|_{L^{p\beta}(\mu)}^{p} \le  C p^{C} \|u\|_{L^p(\mu)}^p, \quad \|u\|_{L^\lambda(\mu)} \le C.$$ 
Then there exists a constant $M$ depending only on $C$ and $\beta$ such that
$$\|u\|_{L^\infty(\mu)} \le M,$$
 where $\|u\|_{L^\infty(\mu)}$ denotes the essential supremum of $u$ with respect to $\mu$. 
\end{lemma}

\proof This is just Moser's iteration.  By hypothesis, one obtains
$$\log \|u\|_{L^{p\beta}(\mu)}\le \frac{\log C}{p} + \frac{C\log p}{ p}+\log  \|u\|_{L^{p}(\mu)}.$$ 
Applying the last inequality successively for $p:=\lambda\beta^l$ with $0 \le l \le k$ and summing up, we get
$$\log \|u\|_{L^{\lambda\beta^k}(\mu)} \le \lambda^{-1}\log C \sum_{j=0}^k \beta^{-j} + C \lambda^{-1}\log \big(\lambda\beta\big) \sum_{j=0}^k (j+1) \beta^{-j}+  \log \|u\|_{L^{\lambda}(\mu)}.$$
Letting $k \to \infty$ gives
$$\log \|u\|_{L^\infty(\mu)} \le C_3,$$
for some constant $C_3>0$ depending only on $C$ and $\beta$. This finishes the proof. 
\endproof

In what follows, the letters $C,C_1,C_2,\ldots$ denote constants depending only on $X, \omega_X, A,K$ and their values can be different from line to line. 
Let $$\omega \in \tilde{\mathcal{V}}(X,\omega_X,n,A,p,K).$$ 
Recall that $\omega$ is not smooth in general, but it is  a closed positive current of bounded potentials. To be precise, we can write $\omega= \theta+ \ddc \psi$, for some smooth form $\theta$ and bounded $\theta$-psh function $\psi$ such that $\|\theta\|_{\mathcal{C}^2}$ is bounded by a constant independent of $\omega$, and $\|\psi\|_{L^\infty}<\infty$.

Write $V_\omega^{-1}\omega^n= f \omega_X^n$, for $f \ge 0$ so that $\int_X f \omega_X^n = 1$ and $\|f\|_{L^1 \log L^p} \le K$. Put $f_1:= \max\{f, \frac{1}{2}\}$. We see that 
$$1= \int_X f \omega_X^n \le \int_X f_1 \omega_X^n = \int_{\{f \ge 1/2\}} f \omega_X^n+ \int_{\{f < 1/2\}} \frac{1}{2} \omega_X^n \le  1+ V_{\omega_X}/2.$$ 
Since $f_1 \ge 1/2$, we can find a smooth function $f_2$ so that $f_2 \ge f_1/2 \ge 1/4$ and 
$$\|f_2\|_{L^1 \log L^p} \le 2 \|f_1\|_{L^1 \log L^p} \le 2K, \quad 1/2 \le \int_X f_2 \omega_X^n \le 1+V_{\omega_X}.$$
Let $\eta$ be the unique closed positive $(1,1)$-current cohomologous to $\omega$ such that 
$$V_\omega^{-1}\eta^n = c f_2 \omega_X^n,$$
where $c:= \frac{1}{\int_X f_2 \omega_X^n}  \in [(1+ V_{\omega_X})^{-1},2]$. 
Observe that 
$$V_\omega^{-1} \omega^n= f \omega_X^n \le f_1 \omega_X^n \le 2f_2 \omega_X^n=  2c^{-1} V_\omega^{-1} \eta^n.$$
Thus we obtain the following crucial inequality
\begin{align}\label{ine-sosanhomegaeta}
\omega^n \le 2c^{-1} \eta^n.
\end{align}
Moreover by \cite[Corollary 4.1]{Guo-Phong-Song-Sturm} (which extends naturally to the case of big and semi-positive classes), there exists a constant $M>0$ independent of $\omega$ such that 
\begin{align}\label{ine-hieucuaomegavaeta}
\omega- \eta= \ddc \tilde{\psi},
\end{align}
for some dsh function $\tilde{\psi}$ with $0 \le \psi \le M$. 

For every constant $\epsilon \in (0,1]$, let $\eta_\epsilon$ be the unique smooth K\"ahler form in $\{\omega\}+ \epsilon \{\omega_X\}$ such that 
$$V_{\eta_\epsilon}^{-1}\eta_\epsilon^n= c f_2 \omega_X^n,$$
 We see that $\eta_\epsilon$ belongs to $\mathcal{W}(X,\omega_X,n,A',p,K', \frac{1}{4(1+V_{\omega_X})})$ (here the constant $\frac{1}{4(1+V_{\omega_X})}$ means the constant function on $X$ whose value is equal to $\frac{1}{4(1+V_{\omega_X})}$) for some constants $A' \ge A,K' \ge K$ depending only on $X,\omega_X, A,p,K$. 
By \cite{Guo-Phong-Song-Sturm2}, there are constants $\beta>1$ (depending only on $p,n$) and $C>0$ such that 
\begin{align}\label{ine-sobolev-def0them}
\bigg(\frac{1}{V_{\omega'}}\int_X |u|^{2 \beta} \omega'^n\bigg)^{\frac{1}{\beta}} \le C \bigg( \frac{1}{V_{\omega'}} \int_X d u \wedge \dc u \wedge \omega'^{n-1}+ \frac{1}{V_{\omega'}}\int_X |u|^2 \omega'^n \bigg),
\end{align}
for every $u \in W^{1,2}(X)$ and every $\omega' \in \mathcal{W}(X,\omega_X,n,A',p,K', \frac{1}{4(1+V_{\omega_X})})$.   

Write $\eta=\theta+\ddc u$ for some bounded $\theta$-psh function $u$ with $\sup_X u=0$, and  $\eta_\epsilon= \theta+ \epsilon \omega_X+ \ddc u_\epsilon$ for some bounded $(\theta+\epsilon \omega_X)$-psh function $u_\epsilon$ with $\sup_X u_\epsilon=0$. The sequence $(u_\epsilon)$ is bounded uniformly in $\epsilon$ (but the bound depends on $\theta$). It is well-known that $u_\epsilon$ converges to $u$ in $L^1$ as $\epsilon \to 0$; see \cite{BEGZ}.
 
\begin{lemma}\label{le-chanduoidayulambda}  For every constant $\delta>0$, there exists a constant $\epsilon_\delta>0$ (depending on $\theta,\delta$) such that for every $0<\epsilon<\epsilon_\delta$, we have $u_\epsilon \ge u -\delta$.
\end{lemma} 

\proof Observe that $u$ is $(\theta+ \epsilon\omega_X)$-psh. Thus we can apply the standard Ko{\l}odziej's capacity estimate to $u, u_\epsilon$ as $(\theta+ \epsilon \omega_X)$-psh functions. Since $f_2 \in L^1 \log L^p$ with $p>n$,  we recall that $\mu:= (\ddc u_\epsilon+\theta+\epsilon\omega_X)^n$ satisfies
\begin{align}\label{ine-capthetaptrun}
\mu(E) \le \big(\capK_{\theta+\epsilon \omega_X}(E)\big)^{1+ (p-n)/n};
\end{align}
see \cite[Remark 1.10]{DDG-family} and also \cite[Chapter 4]{Kolodziej05}. Since $p>n$, combining (\ref{ine-capthetaptrun}) with routine arguments (e.g, see the proof of \cite[Proposition 5.3]{Guedj-Zeriahi-big-stability}) shows that $u- u_\epsilon \le h(\|u_\epsilon -u\|_{L^1})$, where $h:\R_{\ge 0} \to \R_{\ge 0}$ is a continuous function with $h(0)=0$. This combined with the fact that $u_\epsilon \to u$ in $L^1$ gives the desired assertion.
\endproof

\begin{lemma} \label{le-Twedeetaepsi} We have that $\eta_\epsilon^n$  converges weakly to $\eta^n$ as $\epsilon \to 0$. Furthermore, for every closed positive $(1,1)$-current $T$ of bounded potentials, one gets $T\wedge \eta_\epsilon^{n-1} \to T\wedge \eta^{n-1}$ weakly as $\epsilon \to 0$.  
\end{lemma}

\proof The first desired assertion is clear from the definition of $\eta_\epsilon$ and the fact that $V_{\eta_\epsilon} \to V_\eta$ as $\epsilon \to 0$. We check the second desired property.

By Lemma \ref{le-chanduoidayulambda}, we know that for every constant $\delta>0$, there exists a number $\epsilon_\delta>0$ such that  $u_\epsilon \ge u - \delta$ for every $0< \epsilon \le \epsilon_\delta$. Note that we don't claim that the sequence $u_\epsilon$ decreases to $u$, but they are bounded from below by $u-\delta$ in the above sense. The usual continuity properties of Monge-Amp\`ere operators still hold for such a sequence; see, e.g, \cite[Theorem 2.2]{Viet-generalized-nonpluri}. Consequently, we get the desired convergence. 
\endproof

\begin{lemma}\label{le-sobolevchoetaepsilon}
Let $u \in W^{1,2}_*(X)$ be a  positive function such that $u \ge 1$, and   $d u \wedge \dc  u \le T$, for some closed positive $(1,1)$-current $T$ of bounded potentials. Then for every constant $p  \ge 0$, we get
\begin{align}\label{ine-sobolev-def0hai}
\bigg(\frac{1}{V_{\eta}}\int_X u^{p\beta} \eta^n\bigg)^{\frac{1}{\beta}} \le C (p+1)^2 \bigg( \frac{1}{V_{\eta}} \int_X u^{p-2} T \wedge  \eta^{n-1}+ \frac{1}{V_{\eta}}\int_X u^p \eta^n \bigg).
\end{align} 
\end{lemma}

\proof Approximating $u$ by $\min\{u,M\}$ for $M \to \infty$, it suffices to check the desired inequality for $u$ bounded. Since $\eta_\epsilon \in \mathcal{W}(X,\omega_X,n,A',p,K',\frac{1}{4(1+V_{\omega_X})})$ for every constant $\epsilon>0$, by (\ref{ine-sobolev-def0them}), we get 
\begin{align}\label{ine-sobolev-def0}
\bigg(\frac{1}{V_{\eta_\epsilon}}\int_X |v|^{2 \beta} \eta_\epsilon^n\bigg)^{\frac{1}{\beta}} \le C \bigg( \frac{1}{V_{\eta_\epsilon}} \int_X d v \wedge \dc v \wedge \eta_\epsilon^{n-1}+ \frac{1}{V_{\eta_\epsilon}}\int_X |v|^2 \eta_\epsilon^n \bigg),
\end{align}
for every $\epsilon \in (0,1]$ and every $v \in W^{1,2}(X)$, where $C>0$ is independent of $u,v,\omega,p$.
Since $u \ge 1$ and $u$ is bounded,   for every constant $p \ge 0$, we have $u^{p/2} \in W^{1,2}(X)$.  Replacing $v$ by $u^{p/2}$ in (\ref{ine-sobolev-def0}), we get
\begin{align}\label{ine-sobolev-def0up}
\bigg(\frac{1}{V_{\eta_\epsilon}}\int_X u^{p \beta} \eta_\epsilon^n\bigg)^{\frac{1}{\beta}} &\le C(p+1)^2 \bigg( \frac{1}{V_{\eta_\epsilon}} \int_X  u^{p-2} d u \wedge \dc u \wedge \eta_\epsilon^{n-1}+ \frac{1}{V_{\eta_\epsilon}}\int_X u^p \eta_\epsilon^n \bigg)\\
\nonumber
& \le C(p+1)^2 \bigg( \frac{1}{V_{\eta_\epsilon}} \int_X  u^{p-2} T \wedge \eta_\epsilon^{n-1}+ \frac{1}{V_{\eta_\epsilon}}\int_X u^p \eta_\epsilon^n \bigg)
\end{align}
for every $p \ge 0$,  every $\epsilon \in (0,1]$. 
 By Lemma \ref{le-Twedeetaepsi}, one gets that $\eta^n_\epsilon \to \eta^n$ weakly as $\epsilon \to 0$ and  $T\wedge \eta_\epsilon^{n-1} \to T\wedge \eta^{n-1}$ weakly as $\epsilon \to 0$. 
By this, Lemma \ref{le-hoitulientuc} and letting $\epsilon \to 0$ in (\ref{ine-sobolev-def0up}), we obtain (\ref{ine-sobolev-def0hai}) as desired. 
\endproof

We also need the following variant of Lemma \ref{le-sobolevchoetaepsilon}.

\begin{lemma}\label{le-sobolevchoetaepsilonphay}
Let $u \in W^{1,2}_*(X)$ be a positive function such that $d u \wedge \dc  u \le T$, for some closed positive $(1,1)$-current $T$ of bounded potentials. Then for every constant $p \ge 2$, we get
\begin{align*}
\bigg(\frac{1}{V_{\eta}}\int_X u^{p\beta} \eta^n\bigg)^{\frac{1}{\beta}} \le C p^2 \bigg( \frac{1}{V_{\eta}} \int_X u^{p-2} T \wedge  \eta^{n-1}+ \frac{1}{V_{\eta}}\int_X u^p \eta^n \bigg).
\end{align*} 
\end{lemma}

\proof The proof goes exactly as in that of Lemma \ref{le-sobolevchoetaepsilon}, the only difference is that we only have $u^{p/2} \in W^{1,2}_*(X)$ for $p \ge 2$ because we no longer assume that $u \ge 1$. 
\endproof

Let $U_\omega$ be an open dense Zariski subset in $X$ such that $\omega$ is smooth on $U_\omega$.  Let $d_\omega$ be the distance induced by $\omega$ on $U_\omega$.

\begin{lemma} \label{le-L2boundXbinh}  There exist constants $c_0>0,C_1>0$ independent of $\omega$ such that 
$$\int_{U_\omega} e^{c_0 d_\omega(x_0,\cdot)} \omega_X^n \le C_1$$
for some point $x_0 \in U_\omega$ (depending on $\omega$). 
\end{lemma}

As one can see from the proof below, similar inequalities still hold if $d_\omega(x_0,\cdot)$ is replaced by $(d_\omega(x_0,\cdot))^2$. Lemma \ref{le-L2boundXbinh} is, however, sufficient for us.  We refer to \cite{YangLi} for partial results. 

\proof 
Let $v(x_1,x_2):= d_{\omega}(x_1,x_2)$ for $x_1,x_2 \in U_\omega$.
Recall that for $x_1$ fixed, $v(x_1,\cdot)$ is Lipschitz continuous with respect to the distance $d_\omega$ on $U_\omega$. Hence we get
\begin{align}\label{ine-ddcvT0}
\partial_{x_2} v(x_1,\cdot) \wedge \bar{\partial}_{x_2}v(x_1,\cdot) \le \omega
\end{align}
as currents on $U_\omega$. 
We consider local charts $U\Subset U'$ of $X$ biholomorphic to $\B_1,\B_2$ in $\C^n$ respectively, where $\B_r$ denotes the ball centered at $0$ of radius $r$ in $\C^n$. 
\\

\textbf{Claim.}  $\|v\|_{L^2(U^2)}$ is bounded by a constant independent of $\omega$.\\

This claim follows more or less from  the proof of \cite[Lemma 3.1]{Tosatti-noncollapse}, which is in turn based on \cite[Lemma 1.3]{DPS} (see also  \cite[Lemma 2.2]{Rong-Zhang} and \cite[Lemma 4.2]{GGZ-logcontiu}). 
Indeed, the only issue is that $\omega$ is not smooth on $X$  but only on $U_\omega$. Thus when calculating the distance between points in $U_\omega$, only curves lying entirely in $U_\omega$ are considered. This is overcome by observing that the set 
$$\{(x,y) \in U^2: [x,y] \cap (X \backslash U_\omega) \not = \varnothing\}$$
is of zero Lebesgue measure on $U^2$ (one can see it by using Fubini's theorem and the property that the Hausdorff dimension of $X \backslash U_\omega$ is at most $2n-2$), where $[x,y]$ denotes the segment joining $x,y$ (in $\C^n$); see \cite[Lemma 4.2]{GGZ-logcontiu}. The rest of the proof for the claim runs as in  the proof of \cite[Lemma 2.2]{Rong-Zhang}.

By the above claim, we obtain that there exist  constants $r_X>0, C_X>0$ (independent of $\omega$) such that 
\begin{align}\label{ine-BXx0}
\int_{x_0 \in X} \omega_X^n \int_{x \in B_X(x_0,r_X)} v^2(x_0,x) \omega_X^{n} \le C_X,
\end{align}
where $B_X(x_0,r)$ denotes the ball of radius $r$ centered at $x_0$ with respect to the distance $d_{\omega_X}$ on $X$. Let $c_n>0$ be a constant (depending on $\omega_X$) such that 
\begin{align}\label{ine-chonbankinhbe}
r^{-2n}\int_{B_X(x,r)}\omega_X^n \ge c_n   
\end{align}
for every $x \in X$ and $r \in [0, diam(X, \omega_X)]$. Let 
$$B_1:= C_X  8^{2n+2} r_X^{-2n} c^{-1}_n.$$
By (\ref{ine-BXx0}), there exists a Borel subset $E_1$ of $X$ such that
\begin{align}\label{ine-BXx02}
\int_{x \in B_X(x_0,r_X)} v^2(x_0,x) \omega_X^{n} \le B_1, \quad \int_{E_1} \omega_X^n \le C_X B_1^{-1},
\end{align}
for every $x_0 \in X \backslash E_1$ (one observes that $E_1$ depends on $\omega$). Fix now a point $x_0 \in X \backslash E_1$. Let 
$$B_2:=B_1 8^{2n+2} r_X^{-2n}c_n^{-1}.$$
 We deduce from (\ref{ine-BXx02}) that there exists a subset $E_{2,x_0}$ on $X$ such that  
\begin{align}\label{ine-BXx03}
v^2(x_0,x)\le B_2, \quad \int_{E_{2,x_0}} \omega_X^n \le B_1 B_2^{-1},
\end{align}
for every $x \in B_X(x_0,r_X) \backslash E_{2,x_0}$. We now show that 
$$\int_{B_X(x_0, 5r_X/4)} v^2(x_0,x) \omega_X^n$$
 is bounded uniformly in $\omega$ (and in $x_0$). Covering $B_X(x_0,5r_X/4)$ by balls of radius $r_X/4$, we see that it suffices to bound from above $$\int_{x \in B_X(x_1, r_X/4)} v^2(x_0,x) \omega_X^n$$
  uniformly in $\omega,x_0,x_1$ for $x_1$ with $d_{\omega_X}(x_0,x_1) =r_X$. Let $\gamma$ be a minimal geodesic curve joining $x_0,x_1$ (with respect to the smooth metric $\omega_X$). Let $x_1'$ be a point in $\gamma$ so that $d_{\omega_X}(x_0,x'_1)= 7 r_X/8$. Hence $d_{\omega_X}(x'_1,x_1)=r_X/8$.  

By (\ref{ine-chonbankinhbe}) and the choice of $B_1,B_2$,  we infer
$$\int_{B_X(x'_1,r_X/8)} \omega_X^n \ge r_X^{2n} c_n/8^{2n} \ge 2(C_X B_1^{-1}+ B_1 B_2^{-1}).$$
It follows that there exists $x''_1 \in B_X(x'_1,r_X/8)$ such that $x''_1 \not  \in E_1 \cup E_{2,x_0}$. Observe that 
$$d_{\omega_X}(x''_1,x) \le d_{\omega_X}(x''_1,x'_1)+ d_{\omega_X}(x'_1,x_1)+ d_{\omega_X}(x_1, x) \le r_X/8+ r_X/8+ r_X/4= r_X/2,$$
for every $x \in B_X(x_1, r_X/4)$. Consequently, we obtain 
$$\int_{B_X(x_1,r_X/4)} v^2(x''_1,x) \omega_X^n \le \int_{B_X(x''_1,r_X)} v^2(x''_1,x) \omega_X^n \le B_1$$
by (\ref{ine-BXx02}) and the fact that $x''_1 \not \in E_1$. Furthermore since $x''_1 \not \in E_{2, x_0}$, by (\ref{ine-BXx03}), we get $v^2(x_0,x''_1) \le B_2$. Hence we obtain
\begin{align*}
\int_{x \in B_X(x_1, r_X/4)} v^2(x_0,x) \omega_X^n &\le 2 \int_{x \in B_X(x_1, r_X/4)} v^2(x''_1,x) \omega_X^n+ 2 \int_{x \in B_X(x_1, r_X/4)} v^2(x_0,x''_1) \omega_X^n\\
&\le 2(B_1+B_2 V_{\omega_X}).   
\end{align*}
It follows that $$\int_{B_X(x_0, 5r_X/4)} v^2(x_0,x) \omega_X^n$$
 is bounded uniformly in $\omega$ (and in $x_0$) as desired.  Repeating this procedure for $5r_X/4$ instead of $r_X$, we get that for every $m \in \N$,  the quantity $$\int_{B_X(x_0, 5^mr_X/4^m)} v^2(x_0,x) \omega_X^n$$
 is bounded uniformly in $\omega$ and $x_0$. Since $diam(X, \omega_X)$ is finite, for $m$ sufficiently large, the ball $B_X(x_0, 5^mr_X/4^m)$ is equal to the whole $X$. In other words, we have showed that 
 $$\int_{X} v^2(x_0,x) \omega_X^n$$
  is bounded uniformly in $x_0, \omega$  (readers also note that we used only one condition on $\omega$ that the cohomology class $\{\omega\}$ lies in a fixed compact subset in $H^{1,1}(X,\R)$).  
Hence $v(x_0,\cdot)$ is of uniformly $L^2$-norm on $X$. This combined with  \cite[Proposition 3.1]{DS_decay} implies that  $v(x_0,\cdot) \in W^{1,2}_*(X)$ and is of uniform $*$-norm. The desired inequality now follows from \cite[Section 2.2]{Vigny}  or the Moser-Trudinger inequality proved in \cite[Theorem 1.2]{DinhMarinescuVu}.
\endproof

\begin{lemma} \label{le-L2boundXL1log} Let $x_0 \in U_\omega$ be the point in Lemma \ref{le-L2boundXbinh}. Then, there exists a constant $C>0$ independent of $\omega$ such that 
$$V_{\omega'}^{-1} \int_X d_\omega(x_0,\cdot) \omega'^n \le C,$$
for every $\omega' \in \tilde{\mathcal{V}}(X,\omega_X,n,A',p,K')$.
\end{lemma}

\proof We have  $V_{\omega'}^{-1}\omega'^n= f' \omega_X^n$ for $f' \in L^1 \log L^p$ with $\|f'\|_{L^1 \log L^p} \le K$. Let $h:= d_\omega(x_0, \cdot)$. Using Young's inequality (see \cite[Page 111]{Hardy-Littlewood-Polya}), we obtain 
$$\int_X  h f' \omega_X^n \le \int_X e^{c_0 h} \omega_X^n + c_1\|f'\|_{L^1 \log L^p}$$ 
which is bounded by a constant independent of $\omega'$, where $c_1>0$ is a constant depending only on $c_0$ and $\omega_X$. This finishes the proof.
\endproof

\begin{proof}[End of the Proof of Theorem \ref{the-main-univolume}]

Let $x_0 \in U_\omega$ be the point in Lemma \ref{le-L2boundXbinh}.
Put $v:= 1+ d_\omega(x_0, \cdot)$, we have that  $\|v\|_{L^2(\omega_X^n)}$ is bounded uniformly in $\omega$ (one should notice that $v$ is \`a priori not continuous on $X$ and we do not know at this point of the proof that $v$ is bounded). Recall also that $v \in W^{1,2}(X)$ (see \cite{GGZ-logcontiu} or \cite{DS_decay}) and  $dv \wedge \dc v \le \omega$ as currents on $X$. By Lemma \ref{le-L2boundXbinh}, we have that $\|v\|_* \le C_2$ for some constant $C_2$ independent of $\omega$.  Applying (\ref{ine-sobolev-def0hai}) to $v$ shows that
\begin{align}\label{ine-sobolev-def0hai3}
\bigg(\frac{1}{V_{\eta}}\int_X v^{p\beta} \eta^n\bigg)^{\frac{1}{\beta}} \le C (p+1)^2 \bigg( \frac{1}{V_{\eta}} \int_X v^{p-2} \omega \wedge  \eta^{n-1}+ \frac{1}{V_{\eta}}\int_X v^p \eta^n \bigg).
\end{align}
Since $\omega,\eta$ are of bounded potentials, using (\ref{ine-hieucuaomegavaeta}), we can apply Proposition \ref{pro-sobo-etatoT} to the first integral in the right-hand side of (\ref{ine-sobolev-def0hai3}). Thus we obtain
\begin{align}\label{ine-sobolev-def0hai4}
\int_X v^{p-2} \omega \wedge  \eta^{n-1}  &\le C (p+1)^{2(n-1)} \int_X v^{p-2}\omega^n\\
\nonumber
& \le C_1 (p+1)^{2(n-1)} \int_X v^{p-2} \eta^n \le C_1 (p+1)^{2(n-1)} \int_X v^p \eta^n
\end{align}
because of (\ref{ine-sosanhomegaeta}) and the fact that $v \ge 1$. Combining (\ref{ine-sobolev-def0hai4}) with (\ref{ine-sobolev-def0hai3}), we obtain
\begin{align}\label{ine-sobolev-def0hai5}
\bigg(\frac{1}{V_{\eta}}\int_X v^{p\beta} \eta^n\bigg)^{\frac{1}{\beta}} \le C_2 (p+1)^{2n} \frac{1}{V_{\eta}}\int_X v^p \eta^n.
\end{align}
This coupled with Lemma \ref{le-Moseriteration} (for $\lambda=1$) and  Lemma \ref{le-L2boundXL1log} gives 
$$\|v\|_{L^\infty(\eta^n)} \le C_3.$$
 Since a set is of zero $\eta^n$-measure is also of zero $\omega_X^n$-measure, we see that $\|v\|_{L^\infty(\eta^n)}= \|v\|_{L^\infty(\omega_X^n)}$.   Hence the diameter of $\omega$ is bounded uniformly because $v$ is continuous on an open Zariski dense subset on $X$.

It remains to check the local non-collapsing estimate for volumes.
Let $r \in (0,1]$ be a constant. Let $\rho: \R \to \R_{\ge 0}$ be a smooth function such that $\rho= 1$ on $[-1,1]$, $\rho=0$ outside $[-2,2]$ and $0 \le \rho \le 1$. Recall that $U_\omega$ is an open dense Zariski subset in $X$ such that $\omega$ is smooth on $U_\omega$. 
Fix $x \in U_\omega$. We put $u:= d_\omega(x,\cdot)$ and  $u_r:= \rho(d_\omega(x,\cdot)/r)$. Observe that  $0 \le u_r \le 1$ and 
\begin{align}\label{ine-dur}
d u_r \wedge \dc u_r \le B r^{-2} du \wedge \dc u \le B r^{-2} \omega
\end{align}
 for some constant $B\ge 1$ depending only on $\rho$. Let $p \ge 2n$ be a constant. Using (\ref{ine-dur}), (\ref{ine-hieucuaomegavaeta}) and Proposition \ref{pro-sobo-etatoTsuadi} (applied to $u_r$, $T:= B r^{-2} \omega$ and $B r^{-2}\eta$), we see that 
\begin{align}\label{ine-sobolev-def0hai4noncollap}
\int_X u_r^{p-2} \omega \wedge  \eta^{n-1}  &\le C B^{n-1} r^{-2(n-1)} p^{2(n-1)} \int_X u_r^{p-2n} \omega^n  
\end{align}
where we note that $0 \le u_r \le 1$.  Applying Lemma \ref{le-sobolevchoetaepsilonphay} to $u_r$ and using (\ref{ine-sobolev-def0hai4noncollap}), we obtain 
\begin{align}\label{ine-sobolev-def0hai3noncollap}
\bigg(\frac{1}{V_{\eta}}\int_X u_r^{p\beta} \eta^n\bigg)^{\frac{1}{\beta}} \le C p^{2n} r^{-2n}  \frac{1}{V_{\eta}} \int_X u_r^{p-2n} \eta^n,
\end{align} 
 where we use again $0 \le u_r \le 1$. Choosing $p=4n+1$ gives
\begin{align*}  
\bigg(\frac{1}{V_{\eta}}\int_X u_r^{(4n+1)\beta} \eta^n\bigg)^{\frac{1}{\beta}} \le C p^{2n} r^{-2n} \frac{1}{V_{\eta}} \int_X u^{2n+1}_r \eta^n, 
 \end{align*}
This together with (\ref{ine-sosanhomegaeta}) implies
 \begin{align}\label{ine-sobolev-def0hai3noncollap20}
\bigg(\frac{1}{V_{\eta}}\int_X u_r^{(4n+1)\beta} \omega^n\bigg)^{\frac{1}{\beta}} \le C p^{2n} r^{-2n} \frac{1}{V_{\eta}} \int_X u^{2n+1}_r \eta^n, 
 \end{align}
Applying  Proposition \ref{pro-sobo-etatoTsuadi} (applied to $u_r$, $T:= B r^{-2} \omega$) again to the right-hand side (\ref{ine-sobolev-def0hai3noncollap20}), we get 
\begin{align}\label{ine-sobolev-def0hai3noncollap2}
\bigg(\frac{1}{V_{\eta}}\int_X u_r^{(4n+1)\beta} \omega^n\bigg)^{\frac{1}{\beta}} \le C p^{4n} r^{-4n} \frac{1}{V_{\eta}} \int_X u_r \omega^n.
 \end{align}
By this and the definition of $u_r$, we infer that  
 $$\bigg(\frac{vol_\omega(B_\omega(x,r/2))}{V_\omega}\bigg)^{1/\beta} \le C\frac{vol_\omega(B_\omega(x,r))}{V_\omega r^{4n}}$$ 
Hence 
$$\bigg(\frac{vol_\omega(B_\omega(x,r/2))}{V_\omega (r/2)^a} \bigg)^{1/\beta} \le C \frac{vol_\omega(B_\omega(x,r))}{V_\omega r^a},$$
where $a:= 4n\beta/(\beta-1)>0$. 
Applying successively this estimate for $2^{-m} r$ with $m \in \N$, and arguing as in the proof of  \cite[Proposition 9.1]{Guo-Phong-Song-Sturm2}, we obtain 
$$\frac{vol_\omega(B_\omega(x,r))}{V_\omega r^{a}} \ge c_0$$
as desired, for $x \in U_\omega$. Since $U_\omega$ is dense in $\widehat X$, the desired estimate for $x \in \widehat X \backslash U_\omega$ follows by considering a sequence in $U_\omega$ converging to $x$. 
 This finishes the proof.
\end{proof}

As commented in the introduction, we will use the same strategy to prove Theorem \ref{the-volumelocal}.

\begin{proof}[Proof of Theorem \ref{the-volumelocal}] Let $\omega \in \mathcal{M}(X,A_1,x_0, R_0,p_0,K)$ and $V_\omega^{-1}\omega^n= f \omega_X^n$, where $f$ is smooth. Let $0 \le \chi \le 1$ be a cut-off function compactly supported on $B_\omega(x_0,2R_0/3)$ with $\chi =1$ on $B_\omega(x_0, R_0/2)$. Put $f_1:= \chi f \ge 0$. Then  $f_1$ is a smooth function compactly supported on $B_\omega(x_0, R_0)$ such that $f_1=f$ on $B_\omega(x_0, R_0/2)$  and $\int_{B_\omega(x_0, R_0)} f_1 |\log f_1|^{p_0} \omega_X^n$ is uniformly bounded (independent of $x_0,R_0$). Let $\tilde{\omega}$ be the smooth K\"ahler metric cohomologous to $\omega$ such that 
$$V_\omega^{-1}\tilde{\omega}^n= f_1 \omega_X^n+ V_{\omega_X}^{-1}\bigg(1- \int_X f_1  \omega_X^n\bigg) \omega_X^n.$$ 
As above, we can write $\tilde{\omega}- \omega =\ddc \psi$ for some quasi-psh function $\psi$ of $L^\infty$-norm bounded uniformly.  Note that 
\begin{align}\label{ine-sol}
\tilde{\omega}^n \ge \omega^n
\end{align}
 on $B_\omega(x_0,R_0/2)$. 

Let $\rho: \R \to \R_{\ge 0}$ be a smooth function such that $\rho= 1$ on $[-1/2,1/2]$, $\rho=0$ outside $[-2/3,2/3]$ and $0 \le \rho \le 1$. Let $u:=d_\omega(x_0, \cdot)$  and $u_r:=\rho(u/r)$ for $r \le R_0/2$. Note that $u_r$ is supported on $B_\omega(x_0,2R_0/3)$.  We have 
$$d u_r \wedge \dc u_r \le B r^{-2}\omega.$$
Let $p \ge 2n$ be a constant. Using (\ref{ine-sol}) and Proposition \ref{pro-sobo-etatoTsuadi}, we see that 
\begin{align}\label{ine-sobolev-def0hai4noncollapl}
\int_X u_r^{p-2} \omega \wedge  \tilde{\omega}^{n-1}  &\le C B^{n-1} r^{-2(n-1)} p^{2(n-1)} \int_X u_r^{p-2n} \omega^n,
\end{align}
where we note that $0 \le u_r \le 1$.  Applying Lemma \ref{le-sobolevchoetaepsilonphay} to $u_r$ and using (\ref{ine-sobolev-def0hai4noncollapl}), we obtain 
\begin{align}\label{ine-sobolev-def0hai3noncollapl}
\bigg(\frac{1}{V_\omega}\int_X u_r^{p\beta} \tilde{\omega}^n\bigg)^{\frac{1}{\beta}} \le C p^{2n} r^{-2n}  \frac{1}{V_\omega} \int_X u_r^{p-2n} \tilde{\omega}^n,
\end{align} 
 where we use $0 \le u_r \le 1$. Choosing $p=4n+1$ gives
\begin{align*}  
\bigg(\frac{1}{V_\omega}\int_X u_r^{(4n+1)\beta} \tilde{\omega}^n\bigg)^{\frac{1}{\beta}} \le C p^{2n} r^{-2n} \frac{1}{V_\omega} \int_X u^{2n+1}_r \tilde{\omega}^n, 
 \end{align*}
 This combined with (\ref{ine-sol}) gives
 \begin{align}\label{ine-sobolev-def0hai3noncollap20l}
\bigg(\frac{1}{V_\omega}\int_X u_r^{(4n+1)\beta} \omega^n\bigg)^{\frac{1}{\beta}} \le C p^{2n} r^{-2n} \frac{1}{V_{\eta}} \int_X u^{2n+1}_r \tilde{\omega}^n, 
 \end{align}
Applying  Proposition \ref{pro-sobo-etatoTsuadi}  again to the right-hand side (\ref{ine-sobolev-def0hai3noncollap20}), we get 
\begin{align}\label{ine-sobolev-def0hai3noncollap2l}
\bigg(\frac{1}{V_{\eta}}\int_X u_r^{(4n+1)\beta} \omega^n\bigg)^{\frac{1}{\beta}} \le C p^{4n} r^{-4n} \frac{1}{V_{\eta}} \int_X u_r \omega^n.
 \end{align}
At this point we argue as in the end of the proof of Theorem \ref{the-main-univolume} to obtain the desired estimate.
\end{proof}

\begin{proof}[Proof of Corollary \ref{cor-volumelocal}]
Let $C_1,q$ be the constants in Theorem \ref{the-volumelocal}. Recall that they are independent of $x_0,R_0$. Without loss of generality we can assume that $k:=C_1^{-1}$ is an integer.  We show that $\dist_\omega(x, \partial U) \le 6 (k+1)$. Suppose otherwise that $\dist(x, \partial U) > 6(k+1)$. Let $x_j \in B_\omega(x_0,3j+1) \backslash B_\omega(x_0, 3j)$ for $j=1, \ldots, k$. Thus $B_\omega(x_j,2) \subset B_\omega(x_j,3(j+1)) \subset V$ and $B_\omega(x_j,1) \cap B_\omega(x_0,3j-1)=\varnothing$. Hence the balls $B_\omega(x_j,1)$ for $j=0,1, \ldots,k$ are disjoint and contained in $U$. By Theorem \ref{the-volumelocal}, the volume of $B_\omega(x_j,1)$ is at least $V_\omega/k$. It follows that
$$V_\omega \ge vol_\omega(U) \sum_{j=0}^k vol_\omega(B_\omega(x_j,1)) \ge V_\omega (k+1)/k.$$
This is a contradiction. Hence the desired assertion follows. 
\end{proof}

We present now the proof of Corollary \ref{cor-singularsetting}.

\begin{proof}[Proof of Corollary \ref{cor-singularsetting}] Let $\pi: Y \to X'$ be a desingularization of $X'$. Let $\theta_Y$ be a K\"ahler form on $Y$. Write $\pi^* \omega^n_{X'} = g \theta_Y^n$. Then $g$ is a nonnegative smooth function. In particular $M:= \|g\|_{L^\infty}+1$ is finite. Let $ \eta \in \tilde{\mathcal{V}}(X',\omega'_X,n,A,p_0,K)$. We have $\eta^n= f \omega^n_{X'}$. Let $f_1:=f \circ \pi$. We infer that 
\begin{align*}
\int_Y f_1g |\log (f_1g)|^{p_0} \theta_Y^n &= \int_{\{f_1 g \le M\}} f_1 g |\log (f_1 g)|^{p_0} \theta_Y^n+ \int_{\{f_1 g >M\}} f_1 g |\log (f_1 g)|^{p_0} \theta_Y^n\\
 & \lesssim  M \log^{p_0} M \int_Y \theta_Y^n +  \int_{Y} f_1|\log f_1|^{p_0} \pi^* \omega^n_{X'} \lesssim 1.
\end{align*}
Hence the density of $\pi^* \eta^n$ (with respect to $\theta_Y^n$) is of uniform $L^1 \log L^{p_0}$-norm on $Y$. We show now that the cohomology class of $\pi^* \eta$ is bounded uniformly. To this end, we note that since $\pi^* \omega_{X'}$ belongs to big cohomology class, there exists a K\"ahler current $T$ cohomologous to $\pi^* \omega_{X'}$. Hence $T \ge \delta \theta_Y$ for some constant $\delta>0$. By monotonicity of non-pluripolar products (see \cite{Lu-Darvas-DiNezza-mono,Viet-generalized-nonpluri,WittNystrom-mono}), one has 
$$\int_Y \pi^* \eta \wedge \pi^* \omega_{X'}^{n-1} \ge \int_Y \langle \pi^* \eta \wedge T^{n-1} \rangle \ge \delta^{n-1} \int_Y \pi^* \eta \wedge \theta_Y^{n-1},$$  
where we recall that for closed positive $(1,1)$-currents $T_1, \ldots, T_k$ on $Y$, we denote by $\langle T_1 \wedge \cdots \wedge T_k \rangle$ the non-pluripolar product of $T_1, \ldots,T_k$ on $Y$ (see \cite{BEGZ}).  Consequently, the term $\int_Y \pi^* \eta \wedge \theta_Y^{n-1}$ is uniformly bounded. It follows that there are positive constants $A',K'$  so that $\pi^* \eta \in \tilde{\mathcal{V}}(Y,\theta_Y,n,A',p_0,K')$ for every $\eta \in \tilde{\mathcal{V}}(X',\omega'_X,n,A,p_0,K)$. The desired assertions now follow from Theorem \ref{the-main-univolume}.
\end{proof}

\begin{example}\label{ex-no-nonvanishing}  \normalfont Let $X$ be a projective manifold and $(L,h)$ be a very ample Hermitian line bundle on $X$. Let $\omega_X$ be a K\"ahler form on $X$. Fix a constant $A_0>0$. For $s \in H^0(X, L)$ and for $\epsilon \in (0,1]$, let $\omega_{\epsilon,s}$ be the unique K\"ahler form such that $\{\omega_{\epsilon,s}\} \cdot \{\omega_X\}^{n-1} \le A_0$ and 
$$V_{\omega_{\epsilon,s}}^{n-1}\omega_{\epsilon,s}^n= c_{\epsilon,s} (|s|_h^2+\epsilon) \omega_X^n,$$
 where $c_{\epsilon,s}$ is a constant such that the right-hand side of the equation has mass equal to $1$ on $X$. The metrics $\omega_{\epsilon,s}$ forms a family $\mathcal{B}$ which is contained in $\mathcal{V}(X,\omega_X,n,A,p,K)$ for some constants $A,p,K$, but since $L$ is very ample, there is no continuous function $\gamma$ on $X$ with $\dim_{\mathcal{H}}\{\gamma=0\}< 2n-1$ so that  $B$ is a subset of $\mathcal{W}(X,\omega_X,n,A,p,K,\gamma)$ for some constants $A,p,K$. 

 Our main results applies to $\mathcal{B}$ and implies that any metric in $\mathcal{B}$ enjoys a uniform diameter bound and a uniform local non-collapsing estimate for its volume, and hence, $\mathcal{B}$ is relatively compact in the Gromov-Hausdorff topology. We note that one can consider $\omega_{\epsilon,s}$ for $\epsilon=0$ as well. In this case $\omega_{\epsilon,s} \in \tilde{\mathcal{V}}(X,\omega_X,n,A,p,K)$ (because $\omega_{\epsilon,s}$ is a K\"ahler form outside $\{s=0\}$ by \cite{Yau1978}), and the family of $\omega_{\epsilon,s}$ still satisfies similar properties as $\mathcal{B}$.
\end{example}

We mow present the proof of Theorem \ref{the-KEflow}.

\begin{proof}[Proof of Theorem \ref{the-KEflow}]
Let $\Omega$ be a smooth volume form with $\int_X \Omega =1$. Put $\theta:= i \partial \bar \partial \log \Omega$.  Then $\theta$ is a smooth form in $2 \pi c_1(K_X)$. Let $\omega_t:= (1- e^{-t}) \theta+e^{-t} \omega_0$. Thus we can write $\omega(t)= \ddc \varphi+ \omega_t$ and $\varphi(t,x)$ satisfies the equation 
$$\partial_t \varphi = \log \frac{e^{(n- \kappa)t} (i \partial \bar\partial \varphi + \omega_t)^n }{\Omega} - \varphi, \quad \varphi(0, \cdot)=0,$$ 
where $\kappa$ is the numerical dimension of $K_X$, \emph{i.e,} $\kappa$ is the maximal integer $k$ so that $(c_1(K_X))^k \not = 0$. Recall that $\varphi, \partial_t \varphi$ are bounded from above uniformly in $t$, and the volume of $\omega(t)$ is comparable to $e^{(n-\kappa)t}$ (see \cite[Lemmas 9.1 and 9.2]{Guo-Phong-Song-Sturm}). Hence one sees that $\omega(t)$ belongs to $\mathcal{V}(X,\omega_X,n,A,p_0,K)$ for $p_0= 2n$, $A,K$ big enough constants independent of $t$. The desired assertions now follows from Theorem \ref{the-main-univolume}. 
\end{proof}

\bibliography{biblio_family_MA,biblio_Viet_papers,bib-kahlerRicci-flow}

\begin{thebibliography}{10}

\bibitem{Bierstone_Milman}
{\sc E.~Bierstone and P.~D. Milman}, {\em Canonical desingularization in
  characteristic zero by blowing up the maximum strata of a local invariant},
  Invent. Math., 128 (1997), pp.~207--302.

\bibitem{BEGZ}
{\sc S.~Boucksom, P.~Eyssidieux, V.~Guedj, and A.~Zeriahi}, {\em
  Monge-{A}mp\`ere equations in big cohomology classes}, Acta Math., 205
  (2010), pp.~199--262.

\bibitem{Coman-Ma-Marinescu-normal}
{\sc D.~Coman, X.~Ma, and G.~Marinescu}, {\em Equidistribution for sequences of
  line bundles on normal {K}\"{a}hler spaces}, Geom. Topol., 21 (2017),
  pp.~923--962.

\bibitem{Lu-Darvas-DiNezza-mono}
{\sc T.~Darvas, E.~Di~Nezza, and C.~H. Lu}, {\em Monotonicity of nonpluripolar
  products and complex {M}onge-{A}mp\`ere equations with prescribed
  singularity}, Anal. PDE, 11 (2018), pp.~2049--2087.

\bibitem{DPS}
{\sc J.-P. Demailly, T.~Peternell, and M.~Schneider}, {\em Compact {K}\"{a}hler
  manifolds with {H}ermitian semipositive anticanonical bundle}, Compositio
  Math., 101 (1996), pp.~217--224.

\bibitem{DDG-family}
{\sc E.~Di~Nezza, V.~Guedj, and H.~Guenancia}, {\em Families of singular
  {K}\"ahler-{E}instein metrics}.
\newblock \url{arXiv:2003.08178}, 2020.
\newblock to appear in J. Eur. Math. Soc.

\bibitem{DLW}
{\sc T.-C. Dinh, L.~Kaufmann, and H.~Wu}, {\em Dynamics of holomorphic
  correspondences on {R}iemann surfaces}, Internat. J. Math., 31 (2020),
  pp.~2050036, 21.

\bibitem{DLW2}
\leavevmode\vrule height 2pt depth -1.6pt width 23pt, {\em Random walks on
  {${\rm SL}_2(\Bbb C)$}: spectral gap and limit theorems}, Probab. Theory
  Related Fields, 186 (2023), pp.~877--955.

\bibitem{DKC_Holder-Sobolev}
{\sc T.-C. Dinh, S.~a. Ko{\l}odziej, and N.~C. Nguyen}, {\em The complex
  {S}obolev space and {H}\"{o}lder continuous solutions to {M}onge-{A}mp\`ere
  equations}, Bull. Lond. Math. Soc., 54 (2022), pp.~772--790.

\bibitem{Dinh-Ma-Marinescu}
{\sc T.-C. Dinh, X.~Ma, and G.~Marinescu}, {\em Equidistribution and
  convergence speed for zeros of holomorphic sections of singular {H}ermitian
  line bundles}, J. Funct. Anal., 271 (2016), pp.~3082--3110.

\bibitem{DinhMarinescuVu}
{\sc T.-C. Dinh, G.~Marinescu, and D.-V. Vu}, {\em Moser-{T}rudinger
  inequalities and complex {M}onge-{A}mp\`ere equation}, Ann. Sc. Norm. Super.
  Pisa Cl. Sci. (5), 24 (2023), pp.~927--954.

\bibitem{DS_allurepolynom}
{\sc T.-C. Dinh and N.~Sibony}, {\em Dynamique des applications d'allure
  polynomiale}, J. Math. Pures Appl. (9), 82 (2003), pp.~367--423.

\bibitem{DS_decay}
\leavevmode\vrule height 2pt depth -1.6pt width 23pt, {\em Decay of
  correlations and the central limit theorem for meromorphic maps}, Comm. Pure
  Appl. Math., 59 (2006), pp.~754--768.

\bibitem{Do-Vu-log-continuity}
{\sc H.-S. Do and D.-V. Vu}, {\em Log continuity of solutions of complex
  {M}onge-{A}mp\`ere equations}.
\newblock \url{arxiv:2312.04128}, 2023.

\bibitem{DoNguyenW12}
{\sc T.~D. Do and D.-B. Nguyen}, {\em Higher complex sobolev spaces on complex
  manifolds}.
\newblock \url{arXiv:2405.06385}, 2024.

\bibitem{Donaldson-Sun}
{\sc S.~Donaldson and S.~Sun}, {\em Gromov-{H}ausdorff limits of {K}\"{a}hler
  manifolds and algebraic geometry}, Acta Math., 213 (2014), pp.~63--106.

\bibitem{DonaldsonSun2}
\leavevmode\vrule height 2pt depth -1.6pt width 23pt, {\em Gromov-{H}ausdorff
  limits of {K}\"{a}hler manifolds and algebraic geometry, {II}}, J.
  Differential Geom., 107 (2017), pp.~327--371.

\bibitem{Fu-Guo-Song-geometricestimates}
{\sc X.~Fu, B.~Guo, and J.~Song}, {\em Geometric estimates for complex
  {M}onge-{A}mp\`ere equations}, J. Reine Angew. Math., 765 (2020), pp.~69--99.

\bibitem{GGZ-logcontiu}
{\sc V.~Guedj, H.~Guenancia, and A.~Zeriahi}, {\em Diameter of {K}\"ahler
  currents}.
\newblock \url{arXiv:2310.20482}, 2023.

\bibitem{GuedjTo-diameter}
{\sc V.~Guedj and T.~D. T\^o}, {\em K\"ahler families of {G}reen's functions}.
\newblock \url{arXiv:2405.17232}, 2024.

\bibitem{Guedj-Zeriahi-big-stability}
{\sc V.~Guedj and A.~Zeriahi}, {\em Stability of solutions to complex
  {M}onge-{A}mp\`ere equations in big cohomology classes}, Math. Res. Lett., 19
  (2012), pp.~1025--1042.

\bibitem{Guo-Phong-Song-Sturm}
{\sc B.~Guo, D.~H. Phong, J.~Song, and J.~Sturm}, {\em Diameter estimates in
  {K}\"ahler geometry}.
\newblock \url{ https://doi.org/10.1002/cpa.22196}, 2022.
\newblock Comm. Pure Appl. Math.

\bibitem{Guo-Phong-Song-Sturm2}
\leavevmode\vrule height 2pt depth -1.6pt width 23pt, {\em Sobolev inequalities
  on {K}\"ahler spaces}.
\newblock \url{arXiv:2311.00221}, 2023.

\bibitem{GPSS_bodieukien}
\leavevmode\vrule height 2pt depth -1.6pt width 23pt, {\em Diameter estimates
  in {K}\"ahler geometry {II}: removing the small degeneracy assumption}.
\newblock \url{arXiv:2405.18280}, 2024.

\bibitem{Guo-Phong-Sturm-green}
{\sc B.~Guo, D.~H. Phong, and J.~Sturm}, {\em Green's functions and complex
  {M}onge-{A}mp\`ere equations}.
\newblock \url{arXiv:2202.04715}, 2021.
\newblock to appear in Journal of Differential Geometry.

\bibitem{Guo-Phong-Tong}
{\sc B.~Guo, D.~H. Phong, and F.~Tong}, {\em On {$L^\infty$} estimates for
  complex {M}onge-{A}mp\`ere equations}, Ann. of Math. (2), 198 (2023),
  pp.~393--418.

\bibitem{Guo-Phong-Tong-Wang}
{\sc B.~Guo, D.~H. Phong, F.~Tong, and C.~Wang}, {\em On the modulus of
  continuity of solutions to complex {M}onge-{A}mp\`ere equations}.
\newblock \url{arXiv:2112.02354}, 2021.

\bibitem{Guo-Song-localnoncollapsing}
{\sc B.~Guo and J.~Song}, {\em Local noncollapsing for complex
  {M}onge-{A}mp\`ere equations}, J. Reine Angew. Math., 793 (2022),
  pp.~225--238.

\bibitem{Hardy-Littlewood-Polya}
{\sc G.~H. Hardy, J.~E. Littlewood, and G.~P\'{o}lya}, {\em Inequalities},
  Cambridge Mathematical Library, Cambridge University Press, Cambridge, 1988.
\newblock Reprint of the 1952 edition.

\bibitem{Jian-Song}
{\sc W.~Jian and J.~Song}, {\em Diameter estimates for long-time solutions of
  the {K}\"{a}hler-{R}icci flow}, Geom. Funct. Anal., 32 (2022),
  pp.~1335--1356.

\bibitem{Kolodziej05}
{\sc S.~Ko{\l}odziej}, {\em The complex {M}onge-{A}mp\`ere equation and
  pluripotential theory}, Mem. Amer. Math. Soc., 178 (2005), pp.~x+64.

\bibitem{YangLi}
{\sc Y.~Li}, {\em On collapsing {C}alabi-{Y}au fibrations}, J. Differential
  Geom., 117 (2021), pp.~451--483.

\bibitem{Liu-Szekelyhidi2}
{\sc G.~Liu and G.~Sz\'{e}kelyhidi}, {\em Gromov-{H}ausdorff limits of
  {K}\"{a}hler manifolds with {R}icci curvature bounded below {II}}, Comm. Pure
  Appl. Math., 74 (2021), pp.~909--931.

\bibitem{Liu-Szekelyhidi}
\leavevmode\vrule height 2pt depth -1.6pt width 23pt, {\em Gromov-{H}ausdorff
  limits of {K}\"{a}hler manifolds with {R}icci curvature bounded below}, Geom.
  Funct. Anal., 32 (2022), pp.~236--279.

\bibitem{Rong-Zhang}
{\sc X.~Rong and Y.~Zhang}, {\em Continuity of extremal transitions and flops
  for {C}alabi-{Y}au manifolds}, J. Differential Geom., 89 (2011),
  pp.~233--269.
\newblock Appendix B by Mark Gross.

\bibitem{Song-Tian-canonicalmeasure}
{\sc J.~Song and G.~Tian}, {\em Canonical measures and {K}\"{a}hler-{R}icci
  flow}, J. Amer. Math. Soc., 25 (2012), pp.~303--353.

\bibitem{Tian-survey-KE}
{\sc G.~Tian}, {\em Some progresses on {K}\"{a}hler-{R}icci flow}, Boll. Unione
  Mat. Ital., 12 (2019), pp.~251--263.

\bibitem{TianZhang}
{\sc G.~Tian and Z.~Zhang}, {\em On the {K}\"{a}hler-{R}icci flow on projective
  manifolds of general type}, Chinese Ann. Math. Ser. B, 27 (2006),
  pp.~179--192.

\bibitem{Tosatti-noncollapse}
{\sc V.~Tosatti}, {\em Limits of {C}alabi-{Y}au metrics when the {K}\"{a}hler
  class degenerates}, J. Eur. Math. Soc. (JEMS), 11 (2009), pp.~755--776.

\bibitem{Tosatti-collapsing}
\leavevmode\vrule height 2pt depth -1.6pt width 23pt, {\em Adiabatic limits of
  {R}icci-flat {K}\"{a}hler metrics}, J. Differential Geom., 84 (2010),
  pp.~427--453.

\bibitem{Tosatti-KEflow}
\leavevmode\vrule height 2pt depth -1.6pt width 23pt, {\em K{AWA} lecture notes
  on the {K}\"{a}hler-{R}icci flow}, Ann. Fac. Sci. Toulouse Math. (6), 27
  (2018), pp.~285--376.

\bibitem{Tosatti-survey}
{\sc V.~{Tosatti}}, {\em {Collapsing Calabi-Yau manifolds}}, in Differential
  geometry, Calabi-Yau theory, and general relativity. Lectures given at
  conferences celebrating the 70th birthday of Shing-Tung Yau at Harvard
  University, Cambridge, MA, USA, May 2019, Somerville, MA: International
  Press, 2020, pp.~305--337.

\bibitem{Tosatti-surveyKE}
{\sc V.~Tosatti}, {\em Immortal solutions of the {K}\"ahler-{R}icci flow}.
\newblock \url{arXiv:2405.04444}, 2024.

\bibitem{Vigny}
{\sc G.~Vigny}, {\em Dirichlet-like space and capacity in complex analysis in
  several variables}, J. Funct. Anal., 252 (2007), pp.~247--277.

\bibitem{Vigny_expo-decay-birational}
{\sc G.~Vigny}, {\em Exponential decay of correlations for generic regular
  birational maps of {$\Bbb{P}^k$}}, Math. Ann., 362 (2015), pp.~1033--1054.

\bibitem{Vigny-Vu-Lebesgue}
{\sc G.~Vigny and D.-V. Vu}, {\em Lebesgue points of functions in the complex
  {S}obolev space}.
\newblock International Journal of Mathematics, 2023.
\newblock https://doi.org/10.1142/S0129167X24500149.

\bibitem{Vu_nonkahler_topo_degree}
{\sc D.-V. Vu}, {\em Equilibrium measures of meromorphic self-maps on
  non-{K}\"{a}hler manifolds}, Trans. Amer. Math. Soc., 373 (2020),
  pp.~2229--2250.

\bibitem{Viet-generalized-nonpluri}
\leavevmode\vrule height 2pt depth -1.6pt width 23pt, {\em Relative
  non-pluripolar product of currents}, Ann. Global Anal. Geom., 60 (2021),
  pp.~269--311.

\bibitem{Vu-log-diameter}
\leavevmode\vrule height 2pt depth -1.6pt width 23pt, {\em Continuity of
  functions in complex {S}obolev spaces}.
\newblock \url{arXiv:2312.01635}, 2023.

\bibitem{WittNystrom-mono}
{\sc D.~Witt~Nystr\"{o}m}, {\em Monotonicity of non-pluripolar
  {M}onge-{A}mp\`ere masses}, Indiana Univ. Math. J., 68 (2019), pp.~579--591.

\bibitem{Yau1978}
{\sc S.~T. Yau}, {\em On the {R}icci curvature of a compact {K}\"ahler manifold
  and the complex {M}onge-{A}mp\`ere equation. {I}}, Comm. Pure Appl. Math., 31
  (1978).

\end{thebibliography}
\bibliographystyle{siam}

\bigskip

\noindent
\Addresses
\end{document}